\documentclass[reqno]{amsart}
\usepackage{amsmath, amsfonts, amssymb, amsthm, esint, mathtools}
\usepackage[margin=1in]{geometry}
\usepackage{graphicx,xcolor}

\usepackage{cite}
\usepackage[colorlinks,linkcolor = blue,citecolor = green]{hyperref}

\newtheorem{theorem}{Theorem}[section]
\newtheorem{proposition}[theorem]{Proposition}
\newtheorem{corollary}[theorem]{Corollary}
\newtheorem{lemma}[theorem]{Lemma}
\newtheorem{defn}[theorem]{Definition}
\newtheorem{ass}[theorem]{Assumption}
\newtheorem*{ass*}{Assumption}

\newtheorem{remark}[theorem]{Remark}
\newtheorem{theorem*}{Theorem}[section]
\numberwithin{equation}{section}

\newcommand{\T}{\mathbb T}
 \usepackage[T1]{fontenc}

\newcommand{\R}{\mathbb R}
\newcommand{\C}{\mathbb C}

\DeclareMathOperator{\tr}{\mathrm{tr}}

\newcommand{\one}{\mathds{1}}
\newcommand{\eps}{\varepsilon}

\usepackage{comment, color}

\usepackage{dsfont}

 \def \E{\mathbb{E}}

\author{Nestor Guillen}
\address{Courant Institute, New York University, New York, NY 10012}
\email{nestor.g@nyu.edu}
\author{Vladimir A. Kobzar}
\address{Department of Mathematics, The Ohio State University, Columbus, OH 43210}
\email{kobzar.1@osu.edu}

\title{The Koml\'os conjecture for complex discrepancy}
\date{\today}

\begin{document}

\begin{abstract}

The Koml\'os conjecture is a classic problem in discrepancy theory; it asks whether an absolute  constant $K$ exists such that given any $n$ vectors $a_1,\ldots,a_n$ inside the $m$-dimensional Euclidean ball,  regardless of how large $m,n$ are, there is always a selection of signs $\varepsilon_1,\ldots,\varepsilon_n$ guaranteeing  
$$\|\varepsilon_1a_1+\ldots+\varepsilon_na_n\|_\infty \leq K.$$
We show that if the $\varepsilon_i$'s are allowed to take not just the values of $\pm 1$ but any unit modulus complex number, which we refer to as complex discrepancy, then the above inequality holds for a finite, explicit constant $K_{\mathbb{C}}$. Here, the $\ell^\infty$ norm of the resulting vector in $\mathbb{C}^m$ is the largest modulus of its entries, and thus the complex discrepancy of  real vectors is equivalent to their rank-$2$ vector discrepancy.  This  quantity provides an upper bound (modulo uniform constant prefactor) on Gaussian discrepancy -- a discrepancy measure introduced by Chewi, Gerber, Rigollet and Turner. Thus, we also resolve the Koml\'os conjecture for Gaussian discrepancy. Our paper builds upon the recent work of Bansal and Jiang on the Beck-Fiala and Koml\'os conjectures, which we approach from the formalism of Burkholder and the Bellman function method from probability and harmonic analysis. Our work was in part motivated by the realization that the complex discrepancy of the columns of any unitary matrix is equal to 1, a fact that follows from a straightforward calculation based on Idel and Wolf's generalization of the Sinkhorn normal form for unitary matrices.

\end{abstract}

\maketitle

\section{Introduction}\label{sec:intro}

Discrepancy theory encompasses a wide range of problems and methods in combinatorics, geometry and analysis  with applications in computer science and  scientific computing \cite{Matousek99,Travaglini14,Chazelle2000,   CST14}.  One of the main problems in \emph{combinatorial} discrepancy is vector balancing: given vectors $a_1,\ldots,a_n$ in a normed space,  our objective is to choose signs $\varepsilon_1,\ldots,\varepsilon_n \in \{-1,1\}$ in order to minimize
\begin{align*}
  \| \varepsilon_1a_1+\ldots+\varepsilon_n a_n\|.
\end{align*}

Colloquially,  the  significance of this problem  can be illustrated by the following   example \cite{levy26}. Assume that you want to divide  $n$ friends  into 2 teams of roughly equal strength to play a trivia contest covering books and movies.  We can represent the  skills of each player by a vector in $\R^2$. For example, if Alice knows a lot about books and nothing about movies, while  Bob knows nothing about books and a lot about movies, they will be represented by $(1,0)$ and $(0,1)$ respectively. The remaining players may know a bit about both books and movies, and therefore the corresponding vectors will have two strictly positive coordinates.  If there are $m$ trivia subjects, then each player will be described by a vector in $\R^m$; for large $m$ and $n$ the resulting balancing problem appears highly nontrivial.

We will write the  sum  displayed above more compactly as $ A\bar \varepsilon$ where $A$ denotes the $m \times n$ matrix whose columns are given by $a_{1},\ldots,a_{n} $.   The      \emph{discrepancy} of $A$ is defined as:  
\begin{align*}\label{e:disc}
  \text{disc}(A) := \min_{\bar \varepsilon \in \{\pm 1 \}^n} \|A\bar \varepsilon\|_\infty.
\end{align*}
This  classic version of discrepancy is sometimes referred to as \emph{Boolean}  discrepancy.

For given dimensions $m$ and $n$, we let $\mathcal A(m,n)$ represent the  set  of matrices in  $\R^{m \times n}$ with all columns having at most unit Euclidean length, which we will refer to as (real) \emph{Koml\'os  matrices}. The discrepancy of such matrices   is subject to a long-standing conjecture attributed to J\'anos Koml\'os \cite {Spe85}: 
\begin{align*}
 \emph{The Koml\'os conjecture: } K(m,n):= \max_{A\in \mathcal A(m,n)} \text{disc}(A) = O(1) \text{ as }  m,n\to \infty. 
\end{align*}
 Without loss of generality, it is sufficient to resolve this conjecture when $n = m$  \cite{BECK19811},  \cite[Lecture 5]{Spe94}. In the 1990's,  Banaszczyk established  estimates for a broad class of vector balancing problems, obtaining an $O(\sqrt{\log m})$ upper bound for $K(m,n)$ \cite {Ban98}.  Recently, Bansal and  Jiang improved on the state of the art by proving  a $\tilde O(\log ^{1/4} n)$ upper bound in the  Koml\'os setting \cite{BJ2025}.\footnote{The prefactor inside $\tilde O(\cdot)$ is polynomial in $\log\log n$, specifically $(\log\log n)^c$ for some fixed exponent $c>0$ uniformly in $m$. Bansal and Jiang  recently gave an explicit bound with exponent  $c=7/4$ \cite{BJ26}.}

Kunisky  showed that  $K(m,m) \geq 1 +\sqrt{2} -o(1)$, the best known lower bound, using  orthogonal matrices representing the discrete Haar transform \cite{Kun23}.  Pesenti and Vladu  proved the Koml\'os conjecture  for several types of deterministic and random matrices, including  random orthogonal matrices  using an algorithm based on the   $\ell^{1/2}$ regularization of the maximum function \cite{PesentiVladu2023} (see also \cite{pesenti26} for the barrier function interpretation of this regularizer and its connections to the Stieltjes transform). 

While  the Koml\'os conjecture  remained unresolved for many types of feasible matrices, including  (deterministic) orthogonal matrices,    a substantial amount of work was dedicated to studying the analogues of this conjecture for various relaxations of the classic discrepancy. In his proof of Roth's theorem on the discrepancy of the family of arithmetic progressions, Lov\'asz introduced one of the first relaxations of the classic discrepancy, called \emph{vector discrepancy} \cite{Lovasz00}:
 \begin{align*}
  \text{Vdisc}(A) := \min_{u_1, \dots, u_n \in \mathbb{S}^{n-1}} \max_{i \in [m]}\Big|\sum_{j \in [n]} A_{ij}u_j\Big|.
\end{align*}
 This optimization problem admits an equivalent semidefinite programming formulation, and Nikolov showed that an analogue of the Koml\'os conjecture holds in this setting: $\max_{A\in \mathcal A(m,n)} \text{Vdisc}(A) = 1$ \cite{nikolov13}.
 
 Jones and McPartlon introduced \emph{spherical discrepancy}: 
 \begin{align*}
  \text{Sdisc}(A) := \min_{u \in \sqrt {n} \mathbb{S}^{n-1}} \|A u\|_{\infty}.
\end{align*} 
Although this problem is non-convex, they established that an analogue of the Koml\'os conjecture also holds here: $\max_{A\in \mathcal A(m,n)} \text{Sdisc}(A) = O(1)$ using the partial coloring method of Lovett and Meka \cite{LM15, JM20}. (In Section \ref{sec:previous_work}, we provide additional context regarding this method.)

Chewi,  Gerber,  Rigollet and  Turner introduced \emph{Gaussian discrepancy}:
\[
\text{Gdisc}(A) := \min_{\sigma \in \Sigma} \E \| A G \|_\infty
\]
where $G$ is sampled from the Gaussian distribution  $N(0,\sigma)$ and the covariance matrix $\sigma$ belongs to the set $\Sigma$ of  positive semidefinite (PSD)  matrices  in $\R^{n \times n}$ with all ones on the diagonal \cite{CGRT22}. 
They also generalized vector discrepancy defined above to \emph{rank-constrained vector discrepancy}: given any integer $r$ in $[2, n]$: 
\begin{align*}
  \text{Vdisc}_r(A) := \min_{u_1, \dots, u_n \in \mathbb{S}^{r-1}} \max_{i \in [m]}\Big|\sum_{j \in [n]} A_{ij}u_j\Big|.
\end{align*}

In this work, we will focus on the \emph{complex discrepancy} of a matrix $C \in \C^{m\times n}$: 
\begin{align*}
  \text{cdisc}(C) := \min_{z \in \T^n} \|Cz\|_\infty.
\end{align*}
We identify each $z_j = e^{i\phi_j}$  with a unit  $\ell_2$ norm vector $u_j = (\cos \phi_j, \sin \phi_j) \in \R^2$.  Under this identification, for any real  matrix $A \in \R^{m \times n}$,
\begin{align}
 \Big|\sum_{j \in [n]} A_{ij}z_j\Big| = \Big|\sum_{j \in [n]}A_{ij}u_j\Big|, \label{eq:real_complex_identification}
\end{align}
and therefore
\begin{align} \label{eq:cdisc_vdisc2}
  \text{cdisc}(A) =\text{Vdisc}_2(A).
\end{align}
Combining Corollary 5 with Proposition 6 in \cite{CGRT22}, for any  $A \in \R^{m \times n}$, we obtain: 
 \begin{align}
  \text{Vdisc}(A) \vee  \text{Sdisc}(A) \lesssim \text{Gdisc}(A) \leq \sqrt{2} \, \text{Vdisc}_2(A) \leq  \sqrt{2}\, \text{disc}(A).  \label{eq:hierarchy_relaxations}
 \end{align}
 
 Let  $\mathcal C(m,n)$ denote the set of matrices in $ \C^{m \times n}$ with columns of  at most unit $\ell_2$ norm  (we will refer to  matrices in this set as the complex Koml\'os matrices). The Koml\'os conjecture for the complex discrepancy, which we denote as the \emph{complex Koml\'os  (CK) problem}, is 
\begin{align*}
\max_{C\in \mathcal C(m,n)} \text{cdisc}(C) = O(1) \text{ as }  m, n\to \infty. 
\end{align*} 
Our interest in the CK problem  was in part motivated by the realization that we can determine the complex discrepancy of any unitary matrix $U \in U(m)$ by a straightforward calculation based on Idel and Wolf's generalization of the Sinkhorn normal form for unitary matrices \cite{IW15}.

\begin{theorem}\label{thm:unitary}
If $U$ is a unitary $m \times m$ matrix, 
\begin{align*}
\text{cdisc}(U) =1.
\end{align*} 
\end{theorem}

\begin{proof} The  lower bound follows from Remark \ref{rem:elementary_cdisc_bounds}, and the matching upper bound follows from Remark \ref{rem:coisometries}.  
   
\end{proof}
Another straightforward calculation in the  Appendix shows that  to resolve the CK problem, it is sufficient to resolve  the rank-2 Koml\'os problem  for real matrices.
\begin{proposition} \label{eq:complex_vector_equivalence} For the sets of real and complex Koml\'os matrices  $\mathcal A(m,n)$ and  $\mathcal C(m,n)$, defined above, let 
\[
K_2 (m,n):=\max_{A \in \mathcal A(m,n)} \text{Vdisc}_2(A). 
\]
Then, we have
\begin{align*}
K_2(m,n) \leq \max_{C \in \mathcal C(m,n)} \text{cdisc}(C)\leq 2 K_2(2m,n).
\end{align*}
\end {proposition}

Our main result is the following theorem  that resolves the rank-2     Koml\'os problem; it follows from Theorem \ref{thm:Komlos_bound} and Remark \ref{rem:constants}.  Our main conceptual contribution lies in the framing of this  problem within the formalism of Burkholder, which substantially expands the reach of ideas underlying the Bellman function method beyond its original uses in probability and harmonic analysis (see Sections \ref{sec:Bellman} and \ref{sec:Bellman_adaptation} for details about this method and its adaptation to the present problem).

\begin{theorem}
\label{thm:rank_2_Komlos}
If $A$ is a real Koml\'os matrix in $\mathcal A(m,n)$, then 
\begin{align*}
\text{Vdisc}_2(A) \leq 46\sqrt {2}. 
\end{align*}  
\end{theorem}

\begin{remark}
By Eq. \eqref{eq:hierarchy_relaxations}, the preceding theorem also proves  the Koml\'os conjecture for the Gaussian discrepancy.
\end{remark} 

The following corollary follows from Theorem \ref{thm:rank_2_Komlos} and Proposition   \ref{eq:complex_vector_equivalence}.

\begin{corollary}
\label{cor:cdis_Komlos}
If $C$ is a complex Koml\'os matrix in $\mathcal C(m,n)$, 
\begin{align*}
\text{cdisc}(C) \leq 92\sqrt {2}. 
\end{align*}  
\end{corollary}

\begin{remark}
Adjustments to the proof of  Theorem \ref{thm:Komlos_bound}  and Remark \ref{rem:constants} could be introduced to optimize the constant  in Theorem \ref{thm:rank_2_Komlos}, and eliminate the gap between that constant and the one in  Corollary \ref{cor:cdis_Komlos}. For the sake of simplicity, we do not make these adjustments in this work.  
\end{remark}

\subsection{Notation} 
\label{sec:notation} For a vector $x$ in $\C^m$,  $\|x \|_p$ denotes its $\ell_p$ norm, and  $|x |:=\|x\|_2$ specifically denotes the Euclidean ($\ell_2$) norm.  $\T^n$ denotes the $n$-torus in $\C^n$, i.e., the set of all vectors in $\C^n$ with each component having modulus one:
\[
\T^n:= \{(e^{i\phi_1}, \dots, e^{i\phi_n}) \in \C^n ~|~ \phi_j \in \R \}.
\]
For any positive semidefinite Hermitian matrix $H \in \C^{m\times m}$, its square root  $H^{\frac{1}{2}} := U \Lambda^{\frac{1}{2}} U^*$ where the unitary matrix $U$ and the diagonal matrix $\Lambda$ are given by the spectral decomposition $H =: U \Lambda U^* $, and $\Lambda^{\frac{1}{2}}$ is given by the entrywise square root of $\Lambda$. Also for any  matrix $C \in \C^{m\times n}$,  $\sigma_j(C) \geq 0$ denotes its $j$-th largest singular value; we may sometimes refer to the smallest singular value by $\sigma_{\min}(C)$.  

 We denote the set of unitary matrices in $\C^{m \times m}$ by $U(m)$. When $d$ is a vector in $\C^m$ or $\R^m$, $\text{diag}(d)$ refers to a diagonal matrix $D$ in $\C^{m \times m}$ or $\R^{m \times m}$, respectively, with  $D_{ii} = d_i$. When $D$ is a matrix in $\C^{m \times m}$ or $\R^{m \times m}$, $\text{diag}(D)$ refers to a vector $d \in \C^{m}$ or $\R^{m}$, respectively, with $d_i = D_{ii}$.   For vectors $x$ and $y$ in $\C^m$ or $\R^m$, $(x, y) $ denotes their Euclidean inner product; we may include the underlying space in  the subscript whenever it may be useful to emphasize it.  For a matrix $C \in \C^{m  \times m}$,  its operator norm is denoted by $\|C\|$. $I_m$ denotes the identity matrix in $\C^{m\times m}$ or $\R^{m\times m}$.  
 
 The vector $e_j$ represents the $j$-th canonical basis vector of $\R^m$. For a finite set $\mathcal V$ of size $s$, and  vector $(d_j)_{j \in \mathcal V} \in \R^s$,  $\text{diag}_{j \in \mathcal V}(d_j I_2)$  denotes the diagonal matrix in $\R^{2s\times 2s}$ whose $j$-th principal diagonal block is $d_j I_2$.  For two sets $A$ and $B$,  $A\sqcup B$ denotes their disjoint union.

\subsection{Haar system in the $L^p$ spaces and the Bellman function method}  
\label{sec:Bellman}

Our paper builds upon the recent work of Bansal and Jiang on the Beck-Fiala and Koml\'os conjectures, mentioned above, which we approach from the formalism of Burkholder and the Bellman function method from probability and harmonic analysis.  Our description  of this method is based primarily on the lecture notes of Burkholder, as well as the more recent  survey paper and monograph by Os\k{e}kowski \cite{Burkholder89, osekowski2013, osekowski}. 

Let us consider a sequence $( e_k)_{k\geq 0}$ of functions in the real Lebesgue space $L^p$.     The \emph {unconditional constant} $\beta_p(e)$ is the smallest  $\beta$ that satisfies
\begin{align*}
\Big \|\sum_{k=0}^n \eps_k a_k e_k \Big \|_{L^p} \leq \beta 
\end{align*}
uniformly for all sets of real numbers $a_0, \dots, a_{n}$ satisfying $\|\sum_{k=0}^{n} a_k e_k\|_{L^p} =1$ and all choices of  $\eps_k \in \{\pm 1\}$.  Note that this  is  an \emph{(un-)balancing problem} in  $L^p$  where all $a_k$'s and $\eps_k$'s are chosen to maximize the left hand side of the preceding inequality.\footnote{As an example of  another infinite dimensional balancing problem closer to our setting, we remark that Barthe, Gu{\'e}don, Mendelson and  Naor  considered an analogue of the Koml\'os conjecture in the space of the $\ell_\infty$ sequences \cite{BGMN2005}.}

 For $1 \leq p < \infty$,  the standard Haar system $( h_k)_{k\geq 0}$ on the interval $[0,1)$ forms  a basis of $L^p([0,1])$ endowed with Lebesgue measure, and for  $1 <p < \infty$  the Paley-Marcinkiewicz theorem established   the existence of an unconditional constant for this basis \cite{Paley32, Marcinkiewicz1937, Gaposhkin74}. A sequence $d = (d_0, d_1, \dots)$ in $L^p$ is a \emph{martingale difference sequence}  if $d_{n}$ is orthogonal to $\varphi (d_0, \dots, d_{n-1})$  for all bounded functions $\varphi: \R^{n} \rightarrow \R$. %
 
Let $p^* = \max (p,q)$ where $1 <p <\infty$ and $1/p+1/q= 1.$
In his seminal paper \cite{burkholder84}, Burkholder showed that for any sequence of numbers $\eps = (\eps_0, \eps_1   , \dots )$ in $[-1,1]$ and any positive integer $n$, 
\begin{align}
\Big\|\sum_{k=0}^n \eps_k d_k\Big \|_{L^p} \leq (p^*-1) \Big \|\sum_{k=0}^n  d_k \Big \|_{L^p} \label{eq: lp_estimate}
\end{align}
where $(p^*-1)$ is the smallest possible constant.  The Haar system $h$ forms a martingale difference sequence since, for $n \geq 1$, $h_{n}$ has zero mean and is supported  on a set where $\varphi (h_0, \dots, h_{n-1})$ is constant. Therefore,  $\beta_p(h) \leq (p^*-1)$, and in fact this inequality holds with equality \cite{Burkholder89}. 

Burkholder's original proof rests in part  on identifying an explicit solution of a certain nonlinear partial differential equation satisfying suitable boundary  conditions on a domain in $\R^2$ \cite{burkholder84}  (see also  \cite{burkholder1982}). The modern treatment of Burkholder's technique that became  known as the \emph{Bellman function method} or  \emph{Burkholder method} appeared later \cite{Burkholder89}.  We define a real martingale  $f= (f_0, f_1,\dots)$ in $L^p$ and its transform  $g= (g_0, g_1,\dots)$   by
\[
g_{k+1}-g_{k} =  \eps_k (f_{k+1}-f_{k})
\]  
where $(\eps_0,\eps_1, \dots )$ is a predictable sequence of signs $\{\pm 1\}$, and $|g_0| \leq |f_0|$. Suppose that a function  $U:\R^2 \to \R$: $(1^\circ)$ majorizes
\[
V(x,y) := |y|^p -\beta_p^p |x|^p,
\]
 $(2^\circ)$ is concave along $t \rightarrow (x+t, y+\eps t)$, and $(3^\circ)$ satisfies $U\leq 0$ whenever $|y| \leq |x|$. Then, the majorization and Jensen's inequality give
\[
\E V(f_n,g_n) \leq \E U(f_n,g_n) \leq \E U(f_0,g_0) \leq 0.
\]
 Functions that satisfy  $(1^\circ)$  - $(3^\circ)$ are called \emph{Burkholder functions}; the least such function is the \emph{Bellman function} of the problem.

There are two main methods used to find the Bellman function in closed form.   The first one is a modern restatement of Burkholder's original proof \cite{osekowski2013, osekowski2015}. This method identifies the least diagonally concave majorant with   a particular $p$-homogeneous  function $U$ satisfying $U(tx,ty) = t^p U(x, y)$ for $t>0$ and the required properties $(1^\circ)$  - $(3^\circ)$.  Analysis of those properties,  together with rescaling and  minimality of the function, leads to an explicit expression for $U$, yielding in the setting of our example the upper bound \eqref{eq: lp_estimate}. This analysis resembles rescaling arguments in PDEs.

\textbf{Burkholder method and Monge-Amp\`ere equation.} The second method  to find the Bellman function entails solving  a free boundary problem for  homogeneous Monge-Amp\`ere equations \cite{VV10} (see also \cite{NT96, SSV08}). In this setting, restricted concavity can be expressed via Hessian inequalities, and degenerate Monge-Amp\`ere equations can be used to construct candidate solutions. These connections motivate our use of Hessian inequalities, even though our proof is a direct variational argument -- we do not solve a Monge–Amp\`ere equation or establish that our discrepancy potential is an optimal Bellman function. 

\textbf{Burkholder method and sufficient statistics.} Foster, Rakhlin and Sridharan used the Burkholder method to  prove adaptive bounds for online supervised learning problems, such as the online matrix completion  \cite{foster17, foster} (see also \cite{rakhlin_burkholder_lecture}). Consider a set $(x_1, y_1), \dots, (x_n, y_n)$ of training data where each $x_t$ and $y_t$ belong to the input and output spaces,  $\mathcal X$ and $\mathcal Y$, respectively. Online supervised learning is a multiperiod decision problem where the learner receives the training data sequentially  from the data-generating process  and  in each period $t =1, \dots, n$  makes a prediction $\hat y_t$ of $y_t$ based on $x_t$ and $(x_1, y_1), \dots, (x_{t-1}, y_{t-1})$. (This contrasts the classic, offline learning where all the training data is revealed before the learner makes the first prediction.)  For a given loss function $l(\hat y_t, y_t)$,  and a  map $f$ from $\mathcal X$ to $\mathcal Y$ that serves as an evaluation benchmark, the learner's goal is to minimize the  \emph{regret} $R_n$, given by the difference between the learner's and the benchmark's cumulative losses.  The Burkholder method  provides
\[
R_n(f) \leq  \Phi(f, x_1, \dots, x_n)  
\]
where $\Phi: \mathcal F \times \mathcal X^n \to \R$  gives an adaptive, data-dependent upper bound. Since this bound depends on the inputs, it can be meaningful whether the data is  predetermined, sampled from a distribution or chosen adversarially.  The  bound is achieved by an adaptation of the Bellman function method: the authors use a suitable final value function  $V(\tau)$ and a Burkholder function $U(\tau)$, which  satisfies properties analogous to $1^\circ$ -$3^\circ$ and depends on the \emph {sufficient statistic} $\tau$ -   a compressed representation of the inputs $x_1, x_2, \dots $, predictions $\hat y_1, \hat y_2, \dots $, and ground truth outputs  $ y_1,  y_2, \dots$.

\subsection{Previous work on complex  discrepancy} 
\label{sec:previous_work}

While we did not identify  previous work specifically on the Koml\'os conjecture for complex discrepancy,\footnote{We did identify a question about this conjecture on MathOverflow \cite{mathoverflow_CK}.} there is  recent work on the complex analogue of Spencer's ``six standard deviations'' theorem.

   First, let us describe the original real-valued Spencer problem: it entails coloring the  elements of a set $U$ of size $n$ with binary colors $\{\pm 1\}^n$  so that  each member of a given collection  of $m$ subsets of $U$ is colored in the most balanced way possible. This collection can be represented by the incidence matrix  $A \in \{0,1\}^{m\times n}$, and the balance of coloring is measured by $\text{disc}(A)$. The \emph{set coloring problem} considers a closely related class of matrices: $A \in \R^{m \times n }$ and all $|A_{ij}|\leq 1$.  Spencer established that when $m\geq n$, \[
\text{disc (A)} =O \big(\sqrt {n \log( 2m /n) }\big) 
\]by a partial coloring method using the binary entropy potential \cite {Spe85}. This result was also obtained by  Gluskin  using a different method also based on partial coloring \cite{Glushkin89}, and Giannopoulos provided a simplified version of Gluskin's proof  \cite{Giannopoulos1997}.  When $n \geq m$,  this  upper bound can be taken to be   $5.32 \sqrt {m}$,  which led to its  ``six standard deviations suffice" characterization. (Pesenti and Vladu \cite{PesentiVladu2023} improved the constant prefactor in the set discrepancy setting from 5.32 to 4.1.) A matching $\Omega (\sqrt{m})$ lower bound in this setting is achieved by Hadamard matrices \cite{Chazelle2000, AS2000}.

Bansal developed the first polynomial-time  algorithm  that achieves Spencer's upper bound using  partial coloring by a random walk with covariance determined by semidefinite programming \cite{Bansal10}. Subsequently, Lovett and Meka developed a simpler algorithm  that achieved the same upper bound using  a random walk inside the intersection of the unit hypercube with a polytope in $\R^n$ encoding certain discrepancy constraints  \cite{LM15}. 

Kobos and Varivoda considered the complex discrepancy of Spencer matrices $\mathcal {P}(n)$, i.e., square matrices $C$ in $\C^{n \times n}$ with all $|C_{ij}| \leq 1$ \cite{kv26}, and  conjectured that  for all $n \geq 1$
\begin{align*}
\max_{C\in \mathcal  {P}(n)} \text{cdisc}(C) \leq {\sqrt {n}} 
\end{align*}
proving this conjecture for $n=2$ and 3. They also described a number of famous  problems in the geometry of complex vectors that, like the complex Koml\'os conjecture, present a sharp contrast with the real case, where the answers are far less clear.  

Finding a \emph{complex unconditional constant} of a sequence in the complex  $L^p_{\C}$ space is another   (un)balancing problem in the complex domain. 
For sequence $( e_k)_{k\geq 0}$ of functions in the complex Lebesgue space $L_{\C}^p$,     the complex unconditional constant $\beta_p(e, \C)$ is the smallest  $\beta$ that satisfies
\begin{align*}
\Big \|\sum_{k=0}^n e^{i\phi_k} c_k e_k \Big \|_{L_{\C}^p} \leq \beta 
\end{align*}
uniformly for all sets of complex numbers $c_0, \dots, c_n$ satisfying $\|\sum_{k=0}^n c_k e_k\|_{L_{\C}^p} =1$ and all choices of  $\phi_k \in \R$.  Pe\l{}czy\'{n}ski conjectured that the complex and real unconditional constants of the Haar basis are equal; this conjecture was proven by Burkholder \cite{Pelczynski85, Burkholder88} (see also \cite{Burkholder89}).

Finally, let us observe that the breakthrough solution of the Kadison-Singer problem is premised on its characterization as a \emph{matrix} discrepancy problem over complex Hermitian rank-1 matrices (albeit   here there is a finite number of partitions, in contrast to the continuous ``colorings'' in  complex discrepancy) \cite{Weaver04, MSS15}. Interestingly, however, the Kadison-Singer proof involved the analysis of roots of hyperbolic polynomials, which generalize the determinant and have their origin in the PDE theory \cite{hormander1983}.

\subsection {Open problems}  Our work leads to several open problems, which we will describe in this section.  First, under the Koml\'os column normalization, can bounded rank-2 vector discrepancy be converted into a dimension-free bound on the classic discrepancy?

Second, for any matrix $C \in \C^{m \times n}$, one could define the $k$-th \emph{root-of-unity} discrepancy 
\begin{align*}
\text{Udisc}_k(C) := \min_{z \in (\mu_k)^n} \|Cz\|_\infty
\end{align*}
 where $\mu_k $ is the set of the $k$-th roots of unity:  $\mu_k:= \{1,e^{2\pi i/ k},\dots, e^{2\pi i  (k-1)/ k}\}$. Then, for any real matrix $A$, we have  
 \[
 \text{Udisc}_2(A) = \text{disc}(A), 
 \]
 and for each fixed complex matrix $C$
 \[
 \lim_{k \rightarrow \infty} \text{Udisc}_k(C) = \text{cdisc}(C). 
 \]
 Since $\text{Udisc}_k$ is a relaxation of the classic discrepancy for even $k$, proving the Koml\'os conjecture for $\text{Udisc}_k$  with any even  $k \geq 6$, is a prerequisite for proving the classic  version of this conjecture (for real matrices, $\text{Udisc}_4$ is equivalent to the classic discrepancy modulo a uniform constant  prefactor). A related open problem is then to prove that the Koml\'os conjecture for $\text{Udisc}_k$  with any even  $k \geq 6$, implies the classic Koml\'os conjecture.  

 Third, it would be interesting to understand whether the methods of this paper could help determine the constant prefactor in the complex discrepancy of Spencer matrices,  shed light on other problems in the geometry of complex vectors and/or help resolve any conjectures about the Banach-Mazur distance discussed in \cite{kv26}.

 Lastly, although the potential function and its state variables (i.e., the sufficient statistics) used in this work are not guaranteed to be optimal, they provide an upper bound on the optimal discrepancy. Therefore, given the PDE connections between the Burkholder method and online learning discussed above, as well as recent work at the intersection of discrepancy, online learning and PDEs, it may be fruitful to characterize this  potential as a supersolution to the Monge-Amp\`ere equation or another suitable PDE \cite{kobzar, kobzar_geom, kobzar22, kobzar26, GK26,FJLN26}. It would be also interesting to identify the optimal statistics and potential for the complex or rank-2 discrepancy problem, similarly to the corresponding identification in the context of the Burkholder method \cite{NT96, SSV08,VV08, IV15,Ivanisvili15}. 
 
 The next section addresses the complex discrepancy of unitary matrices,  Section \ref{s:heuristics and overview of the proof}  connects our methods to the recent work of Bansal and Jiang and contextualizes them within the formalism of Burkholder and the Bellman function method. Finally, Section \ref{sec:rank_2_core_proof} contains  the core parts of the proof of the rank-2 Koml\'os conjecture, and the appendices set forth several proofs that were omitted from the main part of the paper. 

 \section {Complex discrepancy of unitary matrices}
In this section, we state the  results of this paper pertaining to the complex discrepancy of unitary matrices in full generality; these results are used to obtain Theorem \ref{thm:unitary} discussed previously. 

\subsection{Optimality of a trivial lower bound on complex discrepancy}

We will describe a trivial lower bound on complex discrepancy, and characterize our main result as the  optimality of this lower bound for unitary matrices. 
 \begin{remark} \label{rem:elementary_cdisc_bounds}
For any  matrix $C \in \C^{m \times m}$, 
\[
\sigma_{\min}(C) \leq \text{cdisc}(C)  
\]
and for any  Koml\'os matrix $C \in \mathcal C(m,m)$
\[
\sigma_{\min}(C) \leq 1.
\]
\end{remark}
\begin{proof}
For any $z \in \T^m$, 
\[
|C z| \geq \sigma_{\min}(C) |z| =  \sigma_{\min}(C)\sqrt {m},
\]
and therefore 
\[
\|C z\|_{\infty} \geq  \frac{|C z|}{\sqrt{m}}  \geq \sigma_{\min}(C).
\]
Also, by Hadamard's inequality, for any square matrix, $C \in \C^{m \times m}$,
\[
\sigma_{\min}(C) \leq |\det (C)|^{\frac{1}{m}} \leq \Big(\prod_{j \in [m]} |C_{:j}|\Big)^{\frac{1}{m}}, 
\]
and for any Koml\'os matrix, the right hand side of the above inequality is upper bounded by 1.
\end{proof}

It is also straightforward to show that this trivial lower bound is optimal for \emph{circulant} matrices, and the eigenvectors of unitary circulant matrices can be chosen \emph{biunimodular} \cite{FR15}.
\begin{remark} \label{rem:circulant unitaries}
For any circulant matrix  $C \in \C^{m \times m}$, there exists an  eigenvector $z$  in $\T^m$, and
\[
\text{cdisc}(C) =\sigma_{\min}(C).  
\]
 If $C$ is also unitary, then $Cz \in \T^m$ for such  eigenvector $z$, and therefore
 \[
\text{cdisc}(C) =1.
\]
\end{remark}
\begin{proof}
It is a standard result that any circulant matrix $C \in \C^{m \times m}$ is diagonalized by a discrete Fourier transform matrix $F$ given by $F_{jk}:= \exp( 2\pi i (j-1)(k-1)/m)$; observe that the rows and columns of $F$ belong to $\T^m$. We have,   
\[
CF  = F \Lambda
\]
where  $\Lambda$ is a diagonal matrix of eigenvalues.  Therefore, taking $z$ to be the column of $F$ corresponding to the  eigenvalue  with the smallest modulus guarantees that $\text{cdisc}(C) = \sigma_{\min}(C)$. When $C$ is unitary, it follows immediately that this quantity is 1.  
\end{proof}

The following generalization of the Sinkhorn normal form for unitary matrices by Idel and Wolf guarantees the existence of a biunimodular vector for \emph{any} unitary matrix \cite{IW15}. 
\begin{theorem}\label{thm:sinkhorn}
For every unitary matrix $U \in \C^{m \times m}$, there exist two diagonal unitary matrices $D_1$ and $D_2$ in $\C^{m \times m}$ such that $W: =D_1UD_2 $ satisfies $\sum_i W_{ij}=1$ and $\sum_j W_{ij}=1$ for all $i, j \in [m]$. 
\end{theorem}
The proof of this theorem is nonconstructive and is based on a reformulation of the matrix decomposition problem in terms of symplectic topology. Using this theorem, we can decompose a unitary matrix $U$ as $D_1^* W D_2^*$ where $W$ is a unitary matrix with all row and column sums equal to 1.  Taking $z = D_2 \one$, we have 
\[
Uz = D_1^* W D_2^* D_2 \one = D_1^* W \one =D_1^* \one = \text{diag}(D_1^*)
\]
where the next to last equality holds by the fact that $W \one = \one$ and the vector $\text{diag}(D^*_1)$ lies in $\T^m$ because $D^*_1$ is a unitary diagonal matrix. This  establishes  the existence of a biunimodular vector of $U$.  
\begin{corollary} \label{cor:torus_mapping}
For every unitary matrix $U \in \C^{m \times m}$, there exists a vector $z\in\mathbb T^m$ such that $Uz\in\mathbb T^m$.
\end{corollary}

 \begin{theorem}\label{t:main theorem}
For any $C \in \C^{m\times n}$ with $m \leq n$, we have
\begin{align*}
  \text{cdisc} (C) \leq \| (CC^*)^{\frac{1}{2}} \|_{\infty \to \infty}   
\end{align*}
where 
\[
\| M \|_{\infty \to \infty}:=\max_i \sum_{j}|M_{ij}|.
\]

\end{theorem}
\begin{proof}

We will  show that there exists $z \in \T^n$ and $w \in \T^m$ such that $Cz = Hw$ for 
\[
H:= (C C^*)^{\frac{1}{2}}.
\] 
Consider the singular value decomposition \[
C=: P S J Q^*
\]
where $P \in U(m)$, $S$ is a real diagonal $m \times m$ matrix with diagonal entries given by $(s_1, \dots, s_r, 0, \dots 0)$ where $s_i \in \R$, $s_i>0$ and $r$ is the rank of $C$, $J := [I_m 0] \in \R^{m \times n} $ and   $Q\in U(n)$. Using this decomposition, we define
\[
B:= P J Q^*.
\]
Note that the rows of $B$ are orthonormal since 
\[
BB^* =P J J^*P^* =I_m, 
\]
and therefore we can complete $B$ to an $n\times n$ unitary matrix $U$. Also, 
\[
H = (C C^*)^{\frac{1}{2}} = P SP^*.
\]
By Corollary \ref{cor:torus_mapping}, there exists $z \in \T^n$ such that 
\[
Uz \in \T^n.
\]
Since the first $m$ coordinates of $Uz$ are equal to $Bz$, we have  $w:=Bz \in \T^m$. Therefore,
\[
Cz = PSJQ^* z = PSP^* PJQ^* z=  (CC^*)^{\frac{1}{2}} Bz = Hw.
\]
We conclude the proof of the theorem by observing that
\[
\|Cz\|_\infty= \|Hw\|_\infty \leq \| H \|_{\infty \to \infty} \|w\|_\infty \leq \| H \|_{\infty \to \infty}. 
\]

\end{proof}

\begin{remark} \label{rem:coisometries} Scaled co-isometries or  tight frames are given by  matrices $C \in \C^{m\times n }$ with $m \leq n$ satisfying $CC^* = \lambda I_m$.  For such matrices, $\text{cdisc} (C) \leq \sqrt{\lambda}$.

\end{remark}
This remark  follows from the previous theorem and  shows that the trivial lower bound in Remark \ref{rem:elementary_cdisc_bounds} is optimal for unitary matrices.

\section{Context and overview of the proofs of the rank-2 Koml\'os conjecture} \label{s:heuristics and overview of the proof}

In this section, we will  connect our proofs of the rank-2 Koml\'os conjecture to the recent work of Bansal and Jiang and contextualize them within the formalism of Burkholder and the Bellman function method. 

\subsection{Adaptation of the Bansal-Jiang  state variables to rank-2 vector discrepancy}
\label{sec:BJ}

The starting point of our analysis of rank-2 vector discrepancy is an adaptation to this setting  of the state variables of the algorithm of Bansal and Jiang for the classic Beck-Fiala problem \cite{BJ2025}. While the random walk $x_t$ used in their algorithm evolves in the unit hypercube $[-1,1]^n$,  the corresponding  fractional coloring domain $\mathcal D$ in our rank-2 vector discrepancy problem is given by a set of vectors $x=(x_1, \dots, x_n) \in  (\mathbb{R}^{2})^n$, each with at most unit $\ell_2$ norm: 
 \begin{align}
 \mathcal D : = (B_2)^n~\text{where} ~B_2:= \{ v \in \R^2 : |v| \leq 1\}. \label{eq:calD}
 \end{align}
We will refer to vectors in $\mathcal D$  as \emph{feasible} vectors.  We also  index the set of \emph{active} vectors by  $\mathcal V \subseteq [n]$. 

\begin{defn} A pair of $x \in \mathcal D$ and $\mathcal V \subset[n]$ is \emph{compatible}  if $|x_j|=1$ for all $j \in [n]$ that are not in $\mathcal V$; we will refer to the corresponding vectors as \emph{inactive}.  
\end{defn}
Thus, for any $x \in \mathcal D$,  the largest and  smallest  compatible sets of active vectors are, respectively:
\[
\mathcal V_{\max} = [n]~~~\text{and}~~~ \mathcal V_{\min} = \{~ j \in [n]~~~:~~~ |x_j|<1~\}.
\]   
 
Let  $a_i$ denote the $i$-th row of $A$, and $a_i(j)$ to denote $A_{ij}$. In reference \cite{BJ2025}, the discrepancy of a row is simply given by a fractional coloring  $\text{disc}_t(i):=(a_i,x_t)$.  On the other hand, when $x \in \mathcal D $, to control the Euclidean norm of 
\[
\sum_{j \in [n]}a_i(j) x_j ~~\text{in}~~ \R^2, 
\]
we will control the absolute value of each coordinate  of this vector. To do this, we specify the set of four test directions in $\R^2$
\[
\Theta_4:= \{\pm e_1, \pm e_2\}. 
\]
We will also use $I:=(i, \xi) \in [m]\times \Theta_4$ to denote a row-direction index.   Then, the \emph{discrepancy} of $I$ at $x$ is
\[
\text{d}_I(x):= \Big (\xi, \sum_{j \in [n]}a_i(j) x_j\Big )_{\R^2}.
\]
 The remaining \emph {energy} $G_i(x)$ of row $i$ at $x$ is given by
\[
G_i(x) := \sum_{j \in [n]}a_i(j)^2 (1-|x_j|^2) 
\]
and the active \emph{size} $q_i(\mathcal V)$ of the row $i$ for given $\mathcal V$ is given by
\begin{align}
q_i(\mathcal V) := \sum_{j \in \mathcal V} a_i(j)^2. \label{eq:q_i}
\end{align}
For brevity, we will use the following notation whenever $\mathcal V$ is fixed:
 \[
 q_i \equiv q_i (\mathcal V). 
 \]
Observe that for each feasible $x$,  a compatible $\mathcal V$, and each $i \in [m]$,
\[
0 \leq G_i(x) \leq q_i.
\]
The energy and size of a row here are analogous to the corresponding quantities in the work of Bansal and Jiang. 

\begin{defn} 
    
\label{def:admissibility}
 Given   $\alpha,\beta, C, \lambda>0$ and a set of active vectors $\mathcal V$, we refer to $x\in  \mathcal D$ as \textbf{$\mathcal V$-admissible} if,  for every row-direction index $I =(i, \xi) \in [m]\times \Theta_4$,  a corresponding (in)equality is satisfied:
 \begin{align*}
 \begin{cases} 
 d_I(x) = 0 & q_i > \alpha ~\textbf{(large row)} \\
 z(d_I(x), G_i(x), q_i) \leq 0 & 0< q_i\leq \alpha~\textbf{(medium row)} \\
 d_I(x) \leq C & q_i =0~\textbf{(small row)} 
 \end{cases} 
 \end{align*}
  where  the function $z$ is given by
\begin{align}
z(d,G,q) := d + \frac{\beta}{\sqrt {q} } G - B(q) \label{eq:z}
\end{align}
 and the barrier term \begin{align}
B(q): = C-\Big(\beta + \frac{2}{\lambda}\Big) \sqrt{q}. \label{eq:Bq}
\end{align} 
\end{defn}

\begin{remark}
   For purposes of evaluating the foregoing admissibility constraints, we use active and inactive $x_j$'s; crucially the  definition of $\mathcal V$-admissibility applies even when $\mathcal V = \emptyset$ and therefore $q_i=0$ for all $i \in [m]$. 
\end{remark}

Our Burkholder-type majorant  $z(d_I(x), G_i(x),  q_i)$  is analogous to the random process  $Z_t(i) = Y_{t}(i) -b$ in the work of Bansal and Jiang, where $Y_t(i) = \text{disc}_t(i) +\beta G_t(i)$ is  discrepancy of a row $i$ regularized with the corresponding energy $G_t(i)$ and $b$ is the target discrepancy.  

If we rescale a row $a_i$ as $\delta a_i$ for $\delta>0$, then  $d_I(x)$ becomes   $\delta d_I(x)$ while  $G_i(x)$ and $q_i$ become $\delta^2 G_i(x)$ and $\delta ^2 q_i$ respectively. Therefore, we normalized $\beta G_i(x)$ by $1/\sqrt {q_i}$ to ensure that this term has the same scale  as  $ d_I(x)$ and $(\beta + 2/\lambda) \sqrt{q}$. 

  As in the work of Bansal and Jiang, our objective is to bound the discrepancy of each medium row by a  constant $B(0) =C$ (when all columns are colored $\mathcal V = \emptyset$ and therefore $q_i = 0$ for all rows); our form of the barrier term  
  \[
B(q) = C-  \omega \sqrt {q}
\]
 for a constant  $\omega>0$ ensures that $d$ and $B(q)-C$ have the same normalization under rescaling of the row.  
    
\begin{remark} If  $\beta, \omega, q >0$, $G\geq 0$ and   $z(d,G,q) \leq 0 $, then 
\[
d \leq B(q)- \frac{\beta}{\sqrt {q} } G \leq C-\omega \sqrt{q}  \leq C.
\] \label{rem:medium_disc}
\end{remark}

In our context, avoiding  \emph{bad} rows,  i.e.,  rows with discrepancy exceeding the given target  in  \cite {BJ2025} corresponds to maintaining   $z(d_I(x), G_i(x),  q_i) \leq 0$ for medium rows (which also  ensures $d_I(x)\leq C$ for such rows). We also maintain $d_I (x)=0$ for large rows and $d_I(x)\leq C$ for small rows.

We would like to construct a function that has a minimizer corresponding to $\mathcal V = \emptyset$ while  guaranteeing that $z \leq 0$ as $q_i$ decreases to zero, and therefore $d_I(x)$ is bounded by $B(0) =C$. We will use the following function to find final vector colorings that attain low discrepancy:  
\begin{align}
\gamma(d, G,q)  = \frac{\alpha}{q} \Phi \Big ( \frac{ z(d,G,q) }{\sqrt {q}}\Big) \label{eq:gamma} 
\end{align}
where 
\begin{align}
\Phi(y):= e^{\lambda y} \label{eq:Phi} 
\end{align} and  $\alpha, \beta, \lambda >0$ are constants, which we will also specify later. The   function $\gamma(d, G,q)$  has a constant ratio between drift in $G$ and diffusion in $d$ because it satisfies
\[
\partial_G \gamma - \frac {\beta}{\lambda}  \partial_{dd} \gamma =0. 
\] This equation fixes the ratio between the positive and negative parts of the Hessian of  $\gamma (d_{I}(x), G_i(x),  q_i )
$
differentiated with respect to $x$ in \eqref{eq:D2gamma}. A similar function $\gamma_t(i) = \lambda \exp(\lambda Z_t(i))$ is used to estimate the number of bad rows  in the work of Bansal and Jiang.

The following two remarks, proved in the Appendix, relate  $C$ and $\omega$  to the other constants (this explains our choice of the barrier $B$ in Eq. \eqref{eq:Bq} and leads to the inequality for $C$ in Assumption \ref{as:C} below).  The first remark guarantees that $z$ satisfies an initial value property analogous to the property $(3^\circ)$ of Burkholder functions. The value of the constant $\omega :=(\beta + 2/\lambda)$ comes from this requirement as well as the requirement that $z$ and $\gamma$ must be monotone decreasing as the size $q$ of the row decreases.

\begin{remark} \label{rem:c} For any $\alpha, \beta, \omega >0$, $0<q \leq \alpha $ and $0\leq G \leq q$, setting
\begin{align}
C > (\omega+ \beta) \sqrt{\alpha}, \label{eq:C_inequality}
\end{align}
 guarantees that $z(0,G,q) <0 $. If, in addition,  $\omega \geq 2/\lambda$, then $\gamma(0,G,q)<1$.
\end{remark}

\begin{remark} \label{rem:omega} For any $\beta, \lambda, q>0$ and $0\leq G \leq q$, setting
\[ 
\omega := \beta + \frac{2}{\lambda}
\]
guarantees that $\partial_q z >0$ and if $z \leq 0$, $\partial_q \gamma \geq 0$. 
\end{remark}

We define the domain
\[
\Omega: = \{ (d,G,q) \in \R^3: 0< q \leq \alpha,~ 0 \leq G \leq q  \},
\]
and the inequality \eqref{eq:C_inequality} becomes the following assumption. 
\begin{ass} \label{as:C} The constants $\alpha, \beta, C, \lambda >0$ satisfy
\begin{align} 
C > \Big(2 \beta  +\frac{2}{\lambda}\Big)\sqrt {\alpha}. \label{eq:C_inequality2}
\end{align}
\end{ass}
\noindent  

Since the Euclidean norm of every compatible $x_j$ for $j \notin \mathcal V$ is equal to 1, 
\[
0 \leq G_i(x) \leq q_i.
\]
Therefore, for every medium row, $(d_I(x), G_i(x), q_i)$  automatically belongs to $\Omega$ when $x$ is compatible. The following lemma restates the preceding remarks on $\Omega$. 

\begin{lemma} \label{lem:monotonicity} On the domain   $\Omega $, 
   the following results hold for $z$ and $\gamma$ defined above:
\begin{enumerate}
    \item[(a)] Under Assumption \ref{as:C},
 $z(0,G,q) <0 $ and $\gamma(0,G,q)<1$;
    \item[(b)] if  $z(d,G,q) \leq 0 $, then $d < C$; and
    \item [(c)] if  $z(d,G,q) \leq 0 $, then $\partial_q z >0$ and $\partial_q \gamma \geq 0$.   
\end{enumerate}

\end{lemma}

\begin{remark} Part (c) in the preceding lemma  establishes the following monotonicity properties of $z$ and $\gamma$ on $\Omega$: if $0 < q_{new} \leq q_{old} \leq \alpha$,  $0 \leq G \leq q_{new}$  and $z(d,G,q_{old}) \leq 0$, then 
\[
z(d,G,q_{new}) \leq z(d,G,q_{old})~\text{and} ~\gamma(d,G,q_{new}) \leq \gamma(d,G,q_{old}).
\]
\end{remark}
\begin{remark} In the remainder of this paper, we will assume that Assumption \ref{as:C} is satisfied. In Remark \ref{rem:constants}, we will verify that this assumption is in fact satisfied for a specific set of constants.
\end{remark}

\subsection{Adaptation of the Bellman function method to rank-2 vector discrepancy}  
\label{sec:Bellman_adaptation} 
We will first consider the upper bound on the discrepancy of an individual row and direction, and then address how they are controlled jointly. In the classic Bellman function method, one constructs a Burkholder function $U$ on a low-dimensional domain that majorizes the  value function $V$ and is concave along all admissible directions. In the present setting, for a given row-direction index $I$, the relevant value function is 
\[
\tilde d_{I}(x) :=  d_{I}(x) -C.
\]
For purposes of this section, including Lemmas \ref{lem:Dz}, \ref{lem:grad_hessians} and \ref{eq:weighted_hessian_decomposition}, we will also fix a nonempty set $\mathcal V$ of active vectors since (as noted previously, the size of a row $q_i \equiv q_i(\mathcal  V)$ depends on $\mathcal V$). We  denote the size of  $\mathcal V$ by\footnote{This is another notational departure from work of Bansal and Jiang where  the corresponding quantity is denoted by $n_t$. Since our method does not entail the construction of a random walk, we do not need a time variable.} 
\begin{align}
s:= |\mathcal V|. \label{eq:s}
\end{align}
We identify with $V:=(\R^ {2})^s $ the affine subset of  $(\R^ {2})^n $ obtained by setting all vectors outside of $\mathcal V$ to be constant and treating them as parameters rather than part of the domain in our variational analysis. \\

\textbf{Medium rows.} We  now define  the set of indices corresponding to medium  rows by
\begin{align}
\mathcal Q_{\text{med}} := \{ (i, \xi)  \in  [m] \times \Theta_4 : 0< q_i \leq \alpha \}.  \label{eq:Q}
\end{align} 
For each   $I =(i,\xi) \in \mathcal Q_{\text{med}}$, we define
\begin{align*}
 z_I(x) \equiv z(d_I(x), G_i(x),  q_i) 
 \end{align*} 
 as a function on $V$. 
The corresponding admissibility  constraint is given by $z_I(x) \leq 0$; to emphasize this inequality is defined on the domain $x =(x_j)_{j \in \mathcal V}$,  with the  remaining inactive vectors $x_j$ for $j   \notin \mathcal V$ treated as parameters. 

We define the set of  indices and normal vectors corresponding to active inequality constraints:
\begin{align}
\mathcal Q^*_{\text{med}}(x) := \{ I\in \mathcal Q_{\text{med}}: z_I(x) =0\}~~\text{and}~~~ N_I:= \nabla z_I (x) ~~\text{for each}~~  I \in \mathcal Q_{\text{med}}^*(x). \label{eq:Q*}
\end{align}
 The gradient $\nabla z_I$ 
 acts on the space of vectors $(\R^2)^s$ in the active set $\mathcal V$; for each $j \in  \mathcal V$,  it is given by the following affine function from  $\R^2$ to $\R^2$: 
\begin{align}
\nabla z_I (x)_j = a_{i}(j) \xi - \frac{2 \beta}{\sqrt {q_i(\mathcal V)}} a_i(j)^2 x_j. \label{eq:E_block}
\end{align}

We can confirm that $z_I$ is essentially a Burkholder-like majorant controlling the discrepancy of an individual row in the direction specified in $I$. It has the corresponding properties closely analogous to the definition of Burkholder functions

$(1^{\circ'})$ \emph{Majorization}: As noted earlier, for every feasible $x \in \mathcal D$, $G_i(x)\geq 0$. Therefore, $\tilde d_{I}(x) \leq z_I(x)$.  

$(2^{\circ'})$ \emph{Restricted monotonicity}:  $z_I$ is concave
because its Hessian is negative semidefinite:%
\begin{align}
 D^2  z_I    =   -2\frac{\beta}{\sqrt {q_i}}  \text{diag}_{j \in \mathcal V}(a_i(j)^2 I_2) \in \R^{2s \times 2s}.\label{eq:D2z}
\end{align}
However, in contrast to the application of the Bellman function method to  martingales, which have zero drift, in our  analysis, we will choose perturbations of $x$ orthogonal to  $\nabla z_I$ to eliminate the first-order directional derivative. We will use the following lemma,  proved in the Appendix by a straightforward computation. 
\begin{lemma} \label{lem:Dz} Given $I = (i,\xi)$ in $\mathcal Q_{\text{med}}$, and $x, v\in(\R^2)^n$ with $x_j = v_j$ for all $j \notin \mathcal V$, we have:
\begin{align}
z_I(v) - z_I(x)  = ( \nabla z_I (x), v-x)_{(\R^2)^s}- \frac{\beta}{\sqrt { q_i}}\sum_{j \in \mathcal V} a_{i}(j)^2 |v_j-x_j|^2. \label{eq:z_evolution}
\end{align}
 \end{lemma}
Thus for an active constraint  $I \in \mathcal Q^*_{\text{med}}(x)$ and  any vector $v$  orthogonal to $\nabla z_I(x)$, for any step size $t \in \R$, 
\begin{align}
z_I(x+tv)  = - \frac{\beta}{\sqrt { q_i }}t^2 \sum_{j \in \mathcal V} a_{i}(j)^2 |v_j|^2 \leq 0. \label {eq:concavity}
\end{align}
This property, which we call \emph{restricted monotonicity},  adapts the restricted concavity of the Bukholder method to our setting.

$(3^{\circ'})$ \emph{Initial condition}:
 Finally, for every medium row, $(d_I(x), G_i(x), q_i)$   belongs to $\Omega$ for all $x \in 
 \mathcal D$ compatible with $\mathcal V$. Therefore, by Lemma \ref{lem:monotonicity}(a),  $z(0, G, q) \leq 0$ under Assumption \ref{as:C}; this property is analogous to $(3^\circ)$ in the Bellman function setting. \\

 \textbf{Large and small rows.} We define the set of  indices corresponding to large  rows by \begin{align} \label{eq:Qlarge}
\mathcal Q_{lg} := \{  (i, \xi)  ~\text{in}~  [m] \times \Theta_4 :  q_i > \alpha \}. 
\end{align}  For any $I =(i, \xi)$ in $\mathcal Q_{lg}$,  the  affine  constraint and its normal are, respectively:  
\begin{align}
d_I(x):= \Big ( \xi, \sum_{j \in \mathcal V} a_i(j) x_j \Big )_{\R^2}+ \Big (\xi, \sum_{j \notin \mathcal V}a_i(j) x_j\Big )_{\R^2} =0 
~~ ~~\text{and}~~~~
N_I := (a_i(j) \xi )_{j \in \mathcal V} \in (\R^2)^s. \label{eq:normal_lg}
\end{align}
Accordingly, for any vector $v \in \R^{2s}$ orthogonal to $N_I$, $d_I(x+tv) = d_I(x)=0$. 

\begin{remark} \label{rem:small_rows}
For a small row  $i$ with $q_i=0$, we have  $a_i(j) = 0$ for all $j \in \mathcal V$.   Therefore, perturbations of $x$ supported on the active coordinates $\mathcal V$ will not affect $G_i(x)$  or  $d_I(x)$ when $I=(i,\xi)$ for all $\xi \in  \Theta_4$. Thus,  small rows only impose  constraints that are constant on the active variables. 
\end{remark}
 
\textbf{Aggregate discrepancy bounds.} Since we need to control the discrepancy of every row in every direction, we cannot control a sum of $z_I$'s due to possible cancellations. We also cannot simply take their maximum because that would impact concavity and differentiability. However, cancellations can be addressed  by exponentiation of $z_I$'s - this approach is also based on the work of Bansal and Jiang. Exponentiation does introduce additional positive terms in the Hessian, but they can be controlled by choosing a suitable subspace of $V$ and bounding the trace of the principal diagonal blocks of the Hessian (Lemmas \ref{lem:elminating_covariance} and \ref{eq:weighted_hessian_decomposition})

 Accordingly, for     $I \in \mathcal Q_{\text{med}}$, we  define
\begin{align*}
   \gamma_I(x) \equiv \gamma(d_I(x), G_i(x),  q_i). 
 \end{align*} 
as a function on $V$. This function measures how close is $z_I(x)$ to 0; when $z_I(x)=0$, we have 
\begin{align}
    \gamma_{I}(x) =\frac{\alpha}{q_i}. \label{eq:gamma_at_zero}
\end{align}  The Hessian $D^2 \gamma_I(x)$  is a function of $x \in(\R^2)^n$ and
 acts on the space of vectors $(\R^2)^s$ in the active set $\mathcal V$; it  is given by the following lemma,  verified in the Appendix.

\begin{lemma} \label{lem:grad_hessians}  For any  $I = (i, \xi)$ in $\mathcal Q_{\text{med}}$ and $x\in V$, the Hessian acting on the active variables is 
\begin{align}
D^2 \gamma_I (x)  =\frac{\lambda}{ \sqrt {q_i} }  \gamma_I (x)  \Big(\frac{\lambda}{\sqrt{q_i}}   \nabla z_I (x)  \nabla z_I^\top (x) +D^2  z_I\Big) ~\text{for}~ D^2  z_I    =   -2\frac{\beta}{\sqrt {q_i}}  \text{diag}_{j \in \mathcal V}(a_i(j)^2 I_2) \in \R^{2s \times 2s}.\label{eq:Hessian_gamma}
\end{align}

 \end{lemma}

Although $\gamma_I $ is $\log$-concave in the active variables $V$, it need not be concave. However,  since  $D^2 \gamma_I$ is a difference between a PSD matrix of at most rank one and a diagonal PSD matrix,  we can use this structure to guarantee the existence of  directions with nonpositive second directional derivative   (Lemma \ref{lem:block trace}). 

We now specify the  weights attributable to  small, medium and large rows: for any $I=(i,\xi) \in [m]\times \Theta_4$, 
\[
\hat \gamma_{I, \mathcal V}(x):= 
\begin{cases}
0 & q_i =0\\
\gamma_I (x)  &I \in \mathcal Q_{\text{med}} \\
1 &  I \in \mathcal Q_{\text{lg}} 
\end{cases}
\]

In the Bansal-Jiang paper, for each column $j$, $C_j$ denotes the set of rows $i$ with $a_i(j) \neq 0$, and a medium row is called \emph{dangerous} at time $t$, if $(a_i, x_t)$ exceeds a multiple of the constant discrepancy bound.  The following function is analogous to their count of dangerous rows per column: for a given active column $j \in \mathcal V$, we  define
\begin{align}
L^{\mathcal V}_j(x) := \frac{1}{4}\sum_{I \in [m] \times \Theta_4} a_i(j)^2\hat \gamma_{I, \mathcal V}(x). \label{eq:LVjx}
\end{align}
For any $\mathcal V$ (empty or not) and constants $\tau, \eta>0$, which we will  specify later (see \eqref{eq:softmax_constants}), we  define the softmax $\mathcal S_{\mathcal V}$  as a function on $\mathcal D$:
\begin{align}
\mathcal S_{\mathcal V}(x) := \begin{cases}\tau \log \big(\sum_{j \in \mathcal V} e^{L^{\mathcal V}_j(x) /\tau}\big) &\text{if}~ \mathcal V \neq \emptyset \\
0 & \text{if}~ \mathcal V = \emptyset.
\end{cases} ~~~\text{and}~~~\mathcal J_{\mathcal V}(x) :=\mathcal S_{\mathcal V}(x)- \eta \sum_{j \in [n]}|x_j|^2. 
\label{eq:S_J}
\end{align} 
For brevity, we will use the following notation whenever $\mathcal V$ is fixed:
 \[
 L_j(x) \equiv L^{\mathcal V}_j(x),~~ \mathcal S(x) \equiv \mathcal S_{\mathcal V}(x) ~~\text{and}~~\mathcal J(x) \equiv \mathcal J_{\mathcal V}(x).
 \]

We will now upper bound the size of the set 
\begin{align}
\mathcal Q(x) := \mathcal Q^*_{med}(x) \sqcup \mathcal Q_{lg} \label{eq:Qx}
\end{align}
indexing  the  constraints imposed on the active variables,  based on the following assumption.

\begin{ass} \label{as:alpha} The  constants $\alpha>0$ in Definition \ref{def:admissibility} and  $\rho>0$   given by Eq. \eqref{eq:rho} satisfy 
\begin{align}
\frac{4}{\alpha}+ \rho < 1. \label{eq:A2}
\end{align}
\end{ass}

\begin{remark} In the remainder of this paper, we will assume that the preceding assumption  is satisfied. At the end of our paper, we will fix  the specific numerical values of the constants and verify that it is in fact satisfied (Remark \ref{rem:constants}) 
\end{remark}

\begin{remark} \label{rem:alpha} 

Fix a set of parameters $\alpha, \beta, \lambda>0$ (the last two of which fix the value of $\rho$)  satisfying  Assumption  \ref{as:alpha}, and choose $\delta$
in the interval \begin{align}
0<\delta\le\frac{\alpha(1-\rho)}4-1 \label{eq:delta_bounds}
\end{align}
If the foregoing upper bound is strictly positive, i.e., $ 0<\rho<1$ and $
\alpha>4/(1-\rho)$, then $\delta$ satisfies
\begin{align}
\frac{4(1+\delta) }{\alpha}+ \rho \leq 1.
\end{align}
In the proof of Theorem \ref{thm:Komlos_bound}, we will define $\tau$ and $\eta$, and show that every optimal pair $(\mathcal V^*,x^*)$ in the optimization of $\mathfrak J$ with such $\tau$ and $\eta$ over $\mathfrak F$ (see Eq. \eqref{eq:frakF})  satisfies
$\mathcal S_{\mathcal V^*}(x^*)\le 1+\delta$.

\end{remark}

\begin{lemma} \label{lem:no_active} For any  $x \in \mathcal D$ compatible with $\mathcal V$, assume that there exists $\delta>0$ that satisfies Eq. \eqref{eq:delta_bounds} and  $\mathcal S_{\mathcal V}(x)\le 1+\delta$. Then,  
\[
|\mathcal Q(x)| \leq  \frac{4 s (1+\delta)}{\alpha}
\]
where the constant $\alpha>0$ is specified in Definition \ref{def:admissibility}. 
\end{lemma}
\begin{proof}
For any  $I \in \mathcal Q^*_{med}(x)$, $z_I(x) = 0$, and therefore $\gamma_{I}(x) =\alpha/q_i$ as we noted previously in \eqref{eq:gamma_at_zero}. Since    $\mathcal S(x)\ge\max_jL_j(x)$, the assumption $\mathcal S(x)\le  1+\delta$ implies $L_j(x)\leq  1+\delta$ for every $j$.  Using these facts as well as the definition of $q_i$ in \eqref{eq:q_i} and the fact that for a large row $i$, we have $q_i >\alpha$, 
\begin{align*}
 |\mathcal Q(x)| &= \sum_{(i, \xi) \in \mathcal Q^*_{med}(x)  }  \sum_{j \in \mathcal V} \frac{a_i(j)^2 }{q_i}+ \sum_{(i, \xi) \in \mathcal Q_{lg}  }  \sum_{j \in \mathcal V} \frac{a_i(j)^2 }{q_i} \\
  &\leq  \frac{1}{\alpha}\Big(\sum_{j \in \mathcal V} \sum_{(i, \xi) \in \mathcal Q^*_{med}(x)  }  a_i(j)^2 \hat \gamma_{I,\mathcal V}  +\sum_{j \in \mathcal V} \sum_{(i, \xi) \in \mathcal Q_{lg}  }  a_i(j)^2 \hat \gamma_{I,\mathcal V} \Big)\\
  & \leq  \frac{4}{\alpha} \sum_{j \in \mathcal V} L_j(x)  \leq  \frac{4 s (1+\delta)}{\alpha}.
\end{align*}
\end{proof}

When $s \geq 2$, combining  $L_j$’s through the softmax contributes an extra nonnegative term (the rightmost term in Eq. \eqref{eq:D2S}) to $D^2 \mathcal S$. However, this  term can be eliminated by equalizing the directional derivatives of  $L_j$; this results in at most $s-1$ further affine constraints. 

\begin{lemma}
\label{lem:elminating_covariance}
We fix an arbitrary $j_0 \in \mathcal V$ and define the following subspaces of $V$
\[
  \ker(\nabla L_j(x) -\nabla L_{j_0}(x)): = \{ y \in V : ( \nabla L_j(x) -\nabla L_{j_0}(x), y)_{(\R^2)^s} =0\}  
\] and
\begin{align}
Y_1 :=  \bigcap_{j \in \mathcal V \setminus \{j_0\}} \ker(\nabla L_j(x) -\nabla L_{j_0}(x)). \label{eq:Y1}
\end{align}
with an empty intersection (if $s=1$) equal to $V$.
If $v \in Y_1$, then we have  
\begin{align}
v^\top D^2 \mathcal S(x) v=  v^\top \Big(\sum_{j \in \mathcal V} \varphi_j(x) D^2 L_j(x) \Big) v
\label{eq:quadratic}
\end{align}
for the  softmax weights $\varphi_j$ given for each $j \in \mathcal V$ by 
\begin{align}
\varphi_j(x):= \frac{e^{L_j(x)/\tau }}{\sum_{r\in \mathcal V} e^{L_r(x)/\tau }}.  \label{eq:softmax_weights}
\end{align}

\end{lemma}
\begin{proof}

The gradient and Hessian of $\mathcal S$ with respect to $x$ are respectively 
\begin{align}
\nabla \mathcal S =  \sum_{j \in \mathcal V} \varphi_j \nabla L_j, ~\text{and}~
D^2 \mathcal S =  \sum_{j \in \mathcal V} \varphi_j D^2 L_j + \frac{1}{\tau} D L^\top (\text{diag} (\varphi) - \varphi \varphi^\top )D L \label{eq:D2S} 
\end{align}
where $DL:  (\R^ 2)^s \to \R^s $ is the Jacobian of $L :=(L_1, \dots, L_s)^\top$ whose  rows are given by  $D L(x)_{j, :} = \nabla^\top  L_j(x)$, and  $\varphi$ is the vector of softmax weights, which forms an element in the probability simplex in $\R^s$ for every $x \in (\R^2)^n$.
\[
\varphi(x):= (\varphi_1(x), \dots, \varphi_s(x))^\top.
\]
When $s=1$, and $Y_1=V$, $\varphi_{j_0}=1$, and $v^\top DL^\top \big(\text{diag} (\varphi) - \varphi \varphi^\top \big)DL v$ is trivially zero. When $s\geq 2$, 
by our construction, for any $v \in Y_1$, $(v, \nabla L_j)=(v, \nabla L_{j_0})$ for all $j \neq j_0$. Therefore, $DL(x)v = c_v \one_s$ for some scalar $c_v$ that depends on $v$. Then, 
\begin{align*} 
 v^\top DL^\top \big(\text{diag} (\varphi) - \varphi \varphi^\top \big)DL v = c^2_v \one_s^\top \varphi - c^2_v \one_s^\top \varphi\varphi^\top \one_s 
 =0. 
\end{align*}
and Eq. \eqref{eq:quadratic} follows.
\end{proof} 

We will define a subspace $Y$ of $V$ given by the intersection of the subspace $Y_1$   with the orthogonal complement $W^\perp $ of the span of   normal vectors corresponding to the active constraints. For $v \in Y$, the  extra positive term  Eq.\eqref{eq:D2S} in $D^2 \mathcal S$ attributable  to combining  $L_j$’s through the softmax is eliminated in the direction of such $v$. Moreover, $z_I(x+tv) \leq 0$ for all $I \in \mathcal Q^*_{med}(x)$ by Eq. \eqref{eq:concavity}.  

Now we will lower bound the dimension of $Y$. We start with the $2s$ dimensions in $V$ and use at most $s-1$ dimensions to eliminate the nonnegative  term in the softmax Hessian mentioned above. We will also ``spend'' $\dim W$ dimensions to maintain orthogonality to the active admissibility constraints. The decisive calculation 
\[
\dim Y \geq 2s - (s-1) - \dim W \geq  1+ \rho s.
\]
in the next Lemma shows that this still leaves   $Y$ sufficiently large.

\begin{lemma} \label{lem:dimY_lb}
We  assume that for a given  $x \in \mathcal D$ compatible with $\mathcal V$,   there exists $\delta>0$ that satisfies Eq. \eqref{eq:delta_bounds} and  $\mathcal S_{\mathcal V}(x)\le 1+\delta$.  Let $Y_1$ be given by Eq. \eqref{eq:Y1} and  let $W \subset V$ denote the span of normals $N_I(x)$  given by Eqs. \eqref{eq:Q*} and \eqref{eq:normal_lg} for all  active (in)equalities $I \in \mathcal Q(x)$ given by  Eq. \eqref{eq:Qx}. Then,  
\begin{align}
\dim Y \geq 1+ \rho s~~~~\text{where} ~~~~ Y: = W^\perp \cap Y_1. \label{eq:Y}
\end{align}
\end {lemma}
\begin{proof}
By  Lemma \ref{lem:no_active},
 \begin{align}
 \dim W \leq |\mathcal Q(x)|  \leq \frac{4 (1+\delta) s}{\alpha}. \label{eq:dimW_ub}
 \end{align}
Since  $Y_1$ is at least $2s - (s-1)$ dimensional, using  the preceding bound \eqref{eq:dimW_ub} and  Assumption \ref{as:alpha}:
\[
\dim Y \geq 2s - (s-1) - \dim W \geq 1 + s\Big(1-\frac{4 (1+\delta) }{\alpha}\Big) \geq 1+ \rho s.
\]

\end{proof}

 For  a nonempty set of active vectors $\mathcal V$ and inactive vectors  $x_j$ with $|x_j|=1$ for all indices $j \in [n], j \notin \mathcal V$, we  define the \emph{$\mathcal V$-admissible set}
\begin{align} \label{eq:FV}
\mathcal F_{ \mathcal V}(x) :=\{ v \in \mathcal D: v_j = x_j ~\text{for all}~ j \notin \mathcal V, v ~\text{is}~  \mathcal V\text{-admissible} \}
\end{align} where the $\mathcal V$-admissibility is  given by Definition \ref{def:admissibility}. Suppose that $x$ is compatible with $\mathcal V$, is $\mathcal V$-admissible, and satisfies $\mathcal S_{\mathcal V}(x)\le1+\delta$, with $\delta$ as in \eqref{eq:delta_bounds}. Lemmas \ref{lem:dimY_lb}, \ref{eq:weighted_hessian_decomposition} and \ref{lem:block trace} then provide a nonzero $v\in Y$ such that $(v,D^2 \mathcal S_{\mathcal V}(x) v)\leq 0$, and therefore there exists  a direction of negative second derivative of  $\mathcal J_{\mathcal V}$. If all active coordinates lie strictly inside their disks, i.e., $|x_j|<1$ for $j \in \mathcal V$, then a sufficiently short two-sided segment in the direction $v$ remains in $\mathcal F_{\mathcal V}(x)$. Consequently, such a point cannot be a local minimizer. This idea is used to prove there exists a global minimizer of  $\mathcal J_{\mathcal V}$ on $\mathcal F_{ \mathcal V}(x)$ with $|x_j|=1$ for some $j \in \mathcal V$ (Theorem \ref{thm:min_principle}).  \\

\textbf{Contraction upon  decrease of number of active vectors.}
Let us define the global optimization domain over all possible active sets, as well as redefine $\mathcal J_{\mathcal V}$ as function on such domain.
\begin{align} \label{eq:frakF}
\mathfrak F :=\{ (\mathcal V,x) :   \mathcal V \subseteq [n], x \in\mathcal F_{ \mathcal V} \}  ~~~~~\text{and}~~~~~ \mathfrak J (\mathcal V,x):=\mathcal J_{\mathcal V}(x) 
\end{align}
where
\[
  \mathcal F_{ \mathcal V} :=\{ x \in \mathcal D: |x_j| = 1 ~\text{for all}~ j \notin \mathcal V, x ~\text{is}~  \mathcal V\text{-admissible} \}. 
\]
 Since each $ \mathcal F_{ \mathcal V}$ is a closed subset of a compact set $\mathcal D$, each $\mathcal F_{ \mathcal V}$ is compact as well. Thus, $\mathfrak F$   has finitely many compact components and therefore is compact.   The point $x =0$ is compatible with $\mathcal V = [n]$ and is  $[n]$-admissible ($d_I(0) = 0$ for large and small rows and $z(0, q_i, q_i)<0$ by Lemma \ref{lem:monotonicity}(a) for medium rows).   
Thus,   $\mathcal F_{[n]}$ is nonempty, hence so is $\mathfrak F$.  For each  $\mathcal V$,  $\mathcal J_{\mathcal V}$ is continuous on  $\mathcal F_{\mathcal V}$. Since the set of possible $\mathcal V$'s is finite and  at least one $\mathcal F_{ \mathcal V}$ is nonempty, there exists a global minimizer of $\mathfrak J$ over  $\mathfrak F$.

Thus, we will minimize $\mathfrak J $ over  $\mathfrak F$,  breaking ties in favor of  $\mathcal V$ with the smallest size.   

\begin{remark} [Tie-breaking rule]
    \label{rem:ties}
 Among all  pairs $(\mathcal V,x)$ of global minimizers of $\mathfrak J$  over  $\mathfrak F$, choose one with minimal $| \mathcal V|$.
 \end{remark}

The following lemma, proved in the Appendix,  shows that freezing one or more coordinates on a disk boundary, i.e., decreasing the size of $\mathcal V$ among the class of compatible active sets is a contraction.  

\begin{lemma}[Contraction] \label{lem:contraction}  Let $\mathcal V' \subsetneq \mathcal V$, and let $x \in \mathcal D$ satisfy $|x_j| = 1$ for all $j \notin  \mathcal V'$. If  Assumption \ref{as:C}  holds and $x$ is $\mathcal V$-\emph {admissible}, then $x$ is $\mathcal V'$-\emph {admissible}, and 
\[
\mathcal J_{\mathcal V'}(x) \leq \mathcal J_{\mathcal V}(x).
\]
\end {lemma}

In summary, the  contraction lemma and the tie-breaking rule  guarantee that the global minimizer lies strictly inside every active disk. The minimum principle then rules out a nonempty active set. Therefore, the minimizing active set is empty, and admissibility guarantees the desired discrepancy bound. Thus, our  argument differs from the classical martingale setting and represents a minimum principle: instead of controlling the expectation of a martingale, we rule out the existence of a minimizer in the interior of the constrained set.

However, since we take advantage of  directions along which $\mathcal J_{\mathcal V}$ has nonpositive second directional derivative, our method resembles Burkholder's Bellman function method. Also, our adaptation  of the state variables of the Bansal-Jiang controlled random walk to the present setting resembles the use of sufficient statistics in the more recent applications of the Burkholder method to online learning by Foster, Rakhlin and Sridharan.

 \section {Proof of Koml\'os conjecture for the rank-2 vector discrepancy}
 \label{sec:rank_2_core_proof}

We will now complete our proof of the Koml\'os conjecture for the rank-2 vector discrepancy using the following sequence of results.
The next lemma establishes a useful decomposition of the restriction of the Hessian $D^2\mathcal S_{\mathcal V}$ to $Y$ and the subsequent Lemma \ref{lem:block trace} exploits this decomposition to guarantee restricted concavity $\mathcal S_{\mathcal V}$ and $\mathcal J_{\mathcal V}$.  Theorem \ref{thm:min_principle} provides a minimum principle for $\mathcal J_{\mathcal V}$ over $\mathcal F_{\mathcal V}(x)$, Lemma \ref{lem:global_min} characterizes the global minimum of $\mathfrak J$ over $\mathfrak F$, Theorem \ref{thm:Komlos_bound}   uniformly bounds rank-2 vector discrepancy, and Remark \ref{rem:constants} provides specific numerical values of the constants that satisfy Assumptions \ref{as:C} and  \ref{as:alpha}.

\begin{lemma} \label{eq:weighted_hessian_decomposition}
For  softmax weights $\varphi_j$ defined by \eqref{eq:softmax_weights},
\[
\sum_{j \in \mathcal V} \varphi_j(x) D^2L^{\mathcal V}_j(x)  =K(x)-D(x),
\]
for the PSD matrices $D(x) = \text{diag}_{j \in \mathcal V}(d_j(x) I_2) $ and $K(x)$ in $\R^{2s \times 2s}$  satisfying  
\begin{align}
\tr K_{jj}(x) \leq \rho d_j(x) ~~\text{for}~ ~
\rho := \frac{\lambda  (1 +2 \beta)^2 }{ 2   \beta}. \label{eq:rho}
\end{align}
for all $x \in \mathcal D$ and  each $j \in \mathcal V$ where $K_{jj}(x) \in \R^{2\times 2}$ is the $j$-th principal diagonal block of $K(x)$.
\end{lemma}
\begin{proof}
Let us set
\[
 \sigma_i(x) = \frac{1}{4}\sum_{ j \in \mathcal V} \varphi_j(x) a_{i}(j)^2.
\]
By Eq. \eqref{eq:Hessian_gamma}, 
\[
 \sum_{j \in \mathcal V} \varphi_j(x) D^2 L^{\mathcal V}_j(x)  =K(x)-D(x)
\]
where 
\[
K(x)  =  \lambda^2\sum_{I = (i,\xi) \in \mathcal Q_{\text{med}}} \gamma_I (x) \frac{\sigma_i(x)}{q_i}  \nabla z_I(x) \nabla z_I(x)^\top ~\text{and}~ D(x) = 2  \lambda  \beta\sum_{I= (i,\xi) \in \mathcal Q_{\text{med}} }\gamma_I (x)\frac{\sigma_i(x)}{q_i} \text{diag}_{j \in \mathcal V}(a_i(j)^2 I_2).
\]
This decomposition shows that $K$ and $D$ are positive semidefinite for all $x \in (\R^2)^n$. 

For each $I \in \mathcal Q_{\text{med}}$ and $j \in  \mathcal V$, by  Eq. \eqref{eq:E_block}, and the fact that $a_i(j)^2 \leq q_i$, $|\xi|=1$, and  $|x_j| \leq 1$, we have
\begin{align*}
| \nabla z_I(x)_j| =| a_{i}(j) \xi - \frac{2 \beta}{\sqrt {q_i}} a_i(j)^2 x_j| \leq |a_{i}(j)|(1+ 2\beta) .  
\end{align*}
Therefore, the trace of the PSD matrix $(\nabla z_I(x)\nabla z_I^\top(x))_{jj}$ in $\R^{2 \times 2} $ of rank at most one is upper bounded by  $a_{i}(j)^2(1+ 2\beta)^2  $, and for
\[
d_j(x) = 2    \lambda  \beta\sum_{I= (i,\xi) \in \mathcal Q_{med}} \gamma_I(x) \frac{\sigma_i(x)}{q_i} a_i(j)^2
\]
we have
\[
\tr K_{jj}(x)  \leq \lambda^2\sum_{I = (i,\xi)\in \mathcal Q_{med}} \gamma_I(x) \frac{\sigma_i(x)}{q_i}   (1 +2 \beta)^2 a_{i}(j)^2 =  \rho d_j (x)
\]
where $\rho$ is given by \eqref{eq:rho}. 

\end{proof}

\begin{lemma} (Restricted concavity) \label{lem:block trace} We fix the set of  active vectors $\mathcal V$ and a constant $\rho>0$, and consider the following matrices in $\R^ {2s \times 2s}$  
\[
D:=\text{diag}_{j \in \mathcal V}(d_j I_2) \succeq 0 ~\text {and}~ K \succeq 0,  
\]
and assume that for all $j \in \mathcal V$
\begin{align}
\tr K_{jj} \leq \rho d_j. \label{eq:trace_bound}
\end{align}
If $ W$ is a subspace of $(\R^2)^s$ with $\dim W > \rho s$, then there exists nonzero $v \in W$ satisfying 
\begin{align}
(v, K v) \leq  (v, D v). \label{eq:quadratic_inequality} 
\end{align}

 \end{lemma}
 \begin{proof}  Since $K$ is PSD, then so is $K_{jj}$. When $d_j =0$, $K_{jj}$ is traceless. These facts imply that  $K_{jj}$ is a matrix of all zeros. A standard fact about PSD matrices guarantees that the entire $j$-th block  row and column vanish. Therefore, 
 \begin{align}
 \text{ker} \, D \subseteq     \text{ker} \, K. \label{eq:kernel_inclusion}
 \end{align}
  If $W \cap \text{ker} \, D$  contains a nonzero vector, then the inequality \eqref{eq:quadratic_inequality} holds with equality. 
  
  In the remainder of this proof, we will assume that  $W \cap \text{ker} \, D =\{ 0\} $.  Let $P_D$ denote the orthogonal projection on $\text{ran} \, D$. Since $K$ is symmetric and  the inclusion \eqref{eq:kernel_inclusion} holds, 
 \[
 K = P_D K P_D.
 \]
The Moore--Penrose pseudoinverse $D^\dagger $ is given by $\text{diag}_{j \in \mathcal V: d_j >0}(\frac{1}{d_j} I_2)$ when $d_j \neq 0$ with a zero $2 \times 2$ block corresponding to each $d_j =0$. We let
 \[
 B_D=  D^{\dagger/2} K D^{\dagger/2}
 \]
 which is also a PSD matrix. Then, since  $\text{ran} \, K \subseteq \text{ran} \, D$,
 we have
 \begin{align}
 K = D^{1/2} B_D D^{1/2}. \label{eq:K_decomposition}
 \end{align}
Since we are  assuming that $W \cap \text{ker} \, D =\{0\}$, $D$ is positive definite on $W$. Therefore, we can choose a $D$-orthonormal basis $w_1, \dots, w_r$ of $W$ where $r = \dim W$, and let $v_i =  D^{1/2} w_i$. The  vectors $v_i$ are  orthonormal in $\text{ran} \, D$. The $2 \times 2$ principal diagonal blocks of $D^{\dagger} K$ are given by $\frac{1}{d_j} K_{jj}$ for $d_j \neq 0$ and zero otherwise. By the bounds in \eqref{eq:trace_bound}, we have
\[
\tr (B_D) = \tr (D^{\dagger} K) = \sum_{j : d_j >0} \frac{ \tr (K_{jj})}{d_j} \leq \rho s < r. 
\]
Using this inequality and the decomposition \eqref{eq:K_decomposition}, we have
\[
\sum_{i=1}^r (w_i, K w_i )= \sum_{i=1}^r (v_i, B_D v_i ) \leq \tr (B_D) < r. 
\]
Therefore, there exists at least one $w_i$ such that $(w_i, K w_i ) < 1$, while, by construction of $w_i$'s, all $(w_i, D w_i)=1$. This shows that Eq. \eqref{eq:quadratic_inequality} holds for $v$ equal to such $w_i$. 

 \end{proof}

\begin{theorem} 
     [Minimum principle]\label{thm:min_principle} For a given nonempty $\mathcal V \subseteq [n]$ and   inactive vectors  $(x_j)_{j \in [n], j \notin \mathcal V}$ with $|x_j|=1$, if any, let $x^*$ be a global minimizer of $\mathcal J_{ \mathcal V}$ defined by Eq. \eqref{eq:S_J}  over $\mathcal F_{ \mathcal V}(x)$ defined by
Eq. \eqref{eq:FV}. Assume that for  $\delta>0$ that satisfies Eq. \eqref{eq:delta_bounds}, we have   $\mathcal S_{\mathcal V}(x^*)\le 1+\delta$. Then, at least one $j \in \mathcal V$ satisfies $|x^*_j|=1$. 
\end{theorem}
\begin{proof} 


We prove this theorem by contradiction. Assume that for all $j \in \mathcal V$, $|x_j^*|<1$.  We  identify with $V=(\R^ {2})^s $ the affine subset of  $(\R^ {2})^n $ obtained by fixing all inactive vectors  $x_j$ with $|x_j|=1$ (if any) where  $s$ is the size of $\mathcal V$  (Eq. \eqref{eq:s}). 
By construction,   $x^*$ belongs to $ \mathcal D$ and is compatible with $\mathcal V$. Therefore, by Lemma \ref{lem:dimY_lb},    
\[
\dim Y  \geq  1+\rho s
\]
where $\rho$  given by \eqref{eq:rho}  and  the subspace $Y$ of $V$ is given by Eq. \eqref{eq:Y}. Also, by our construction of $Y$, for any $v \in Y$  
\begin{align}
v^\top D^2 \mathcal S(x^*) v=  v^\top \Big(\sum_{j \in \mathcal V} \varphi_j(x^*) D^2 L_j(x^*) \Big) v.
\end{align}
for  softmax weights $\varphi_j$ defined by \eqref{eq:softmax_weights}. Then, by Lemma  \ref{eq:weighted_hessian_decomposition}, 
\[
\sum_{j \in \mathcal V} \varphi_j(x^*) D^2L_j(x^*)  =K(x^*)-D(x^*),
\]
for the PSD matrices $D(x^*) = \text{diag}_{j \in \mathcal V}(d_j(x^*) I_2) $ and $K(x^*)$ in $\R^{2s \times 2s}$  satisfying  
\begin{align*}
\tr K_{jj}(x^*) \leq \rho d_j(x^*). 
\end{align*} 
 Therefore, by Lemma \ref{lem:block trace}, we can choose a nonzero vector $v \in Y$ such that $(v,D^2 \mathcal S(x^*) v)\leq 0$, and   therefore
\[
(v, D^2 \mathcal J(x^*) v) \leq - 2\eta  |v|^2 <0.
\]

By construction of $Y$, any $v \in Y$ is orthogonal to the normal vectors  $N_I$ for each $I \in \mathcal Q_{lg}$ associated with the affine constraints corresponding to large rows. Therefore, $d_I(x^*+tv) = d_I(x^*)=0$  for all $t \in \R$ (see Eq. \eqref{eq:normal_lg} and the discussion accompanying it). Also, as discussed previously, $z_I(x^*+tv) \leq 0$ for all $I \in \mathcal Q^*_{med}(x^*)$ (see \eqref{eq:concavity} and the accompanying text). Finally, if an inequality is inactive, i.e., $z_I(x^*)<0$, then by continuity $z_I(x^*+ tv) < 0$ for sufficiently small $|t|$. Similarly, our assumption that for each $j \in \mathcal V$, $|x_j^*|<1$, guarantees that  $|x_j^*+tv_j|< 1$ for sufficiently small $|t|$.   (By Remark \ref {rem:small_rows},  there are no constraints on $x$ in $V$ relating to the admissibility of   small rows, i.e., their admissibility is only determined by the inactive vectors $x_i$ for $i \notin \mathcal V$.) Therefore, $x^* + tv$ belongs to $\mathcal F_{\mathcal V}(x)$ when $|t|$ is sufficiently small.

   However, $g(t) =\mathcal J(x^*+tv)$ has a local minimum at $t=0$, so $g''(0)=v^\top D^2 \mathcal J(x^*) v \geq 0$, which is a contradiction.   Therefore, at least one $j \in \mathcal V$ satisfies $|x^*_j|=1$ .

\end{proof}

\begin{lemma}      [Global minimum] \label{lem:global_min}  Let $(\mathcal V^*, x^*)$ be a pair minimizing  $\mathfrak J$  over  $\mathfrak F $  with ties broken according to Remark \ref{rem:ties}. We    assume that for  $\delta>0$ that satisfies Eq. \eqref{eq:delta_bounds}, we have   $\mathcal S_{\mathcal V^*}(x^*)\le 1+\delta$. Then, we have $\mathcal V^*= \emptyset$. 
\end{lemma}
\begin{proof} Suppose for contradiction that $\mathcal V^* \neq \emptyset$.  The tie-breaking rule (Remark \ref{rem:ties}) and contraction (Lemma \ref{lem:contraction}) imply $|x_j^*|<1$ for every $j\in\mathcal V^*$. Fix the inactive coordinates at $x_j^*$ for $j\notin\mathcal V^*$. Since
\[ x^*\in\mathcal F_{\mathcal V^*}(x^*) \subseteq\mathcal F_{\mathcal V^*}, 
\]
global optimality implies that $x^*$ minimizes $\mathcal J_{\mathcal V^*}$ over $\mathcal F_{\mathcal V^*}(x^*)$. The assumed bound on $\mathcal S_{\mathcal V^*}(x^*)$  therefore permits the application of Theorem \ref{thm:min_principle}.

 Since $\mathcal V^*$ is assumed to be nonempty,  by this theorem, for some $j \in  \mathcal V^*$, we have $|x^*_j| =1$, which is a contradiction. Therefore,  $\mathcal V^* = \emptyset$.

\end{proof}

\begin{theorem} \label{thm:Komlos_bound} 
For every  real Koml\'os matrix $A \in \mathcal A(m,n)$, there exist unit $\ell_2$ norm vectors $x_1, \dots, x_n \in \R^2$ such that
\[
\max_i \Big|\sum_{j \in [n]} A_{ij}x_j\Big|\leq C\sqrt{2}. 
\]
\end{theorem}
\begin{proof} 

By Assumption  \ref{as:alpha} and Remark \ref{rem:alpha},  there exists a   constant $\delta >0$  satisfying 
\begin{align}
\frac{4  (1 +\delta)}{\alpha}+ \rho \leq 1 \label{eq:1plus_delta_bound}
\end{align}
We set the  parameters $\tau$ and  $\eta$ in the definitions of $\mathcal S_{\mathcal V}$ and $\mathcal J_{\mathcal V}$ as follows:
\begin{align}
\eta := \frac{\delta }{2n}~~\text{and}~~
\tau = \frac{\delta}{2\log(e n)}. \label{eq:softmax_constants}
\end{align}
and choose a global minimizer $(\mathcal V^*,x^*)$ of $\mathfrak J$ over  $\mathfrak F$  for these parameters, with ties broken as in Remark \ref{rem:ties}. 

When $x=0$, the only compatible set of active vectors $\mathcal V$ is $[n]$, and therefore $d_{I}(0) =0$ for all $I \in [m] \times \Theta_4$ and $G_i(0) = q_i([n])$ for all $i \in [m]$. Thus,   we have 
\begin{align}
\hat \gamma_{I, [n]}(0) = \begin{cases} 
0&\text{if}~ q_i=0\\
\gamma_I(0) =  \gamma(0, q_i, q_i)~   &\text{if}~ 0< q_i\leq \alpha \\
1&\text{if}~ q_i>\alpha.
\end{cases} \label{eq:hat_gamma0}
\end{align}
Assumption \ref{as:C} and Lemma \ref{lem:monotonicity}(a) guarantee that $\gamma(0, q_i, q_i) <1 $ for all $I \in \mathcal Q_{med}$, and therefore,
\[
L_j^{[n]}(0) \leq \sum_{i \in [m]}a_i(j)^2   \leq 1
\]
for all $j\in [n]$ due to the $\ell_2$ norm bound on the columns of the Komlos matrices.  Then, by our choice of $\tau$,\footnote{We chose $2\log(e n)$ instead of $2\log( n)$ in the denominator of $\tau$ to avoid treating the $n=1$ case differently from the cases with larger $n$ (for $n\geq 2$, our argument also works with $2\log( n)$ in the denominator of $\tau$).}
\[
\mathcal J_{[n]}(0)=\mathcal S_{[n]}(0) \leq \tau \log (n e^{1/\tau}) = 1+ \tau \log (n) \leq  1 + \frac{\delta}{2}.
\]
By optimality of $( \mathcal V^*, x^*)$, 
\begin{align*}
\mathcal J_{ \mathcal V^*}(x^*)  \leq \mathcal J_{ [n]}(0) \leq 1 + \frac{ \delta}{2}.
\end{align*}
Since $\eta n \leq \delta/2$, we have
\begin{align*}
\mathcal S_{\mathcal V^*}(x^*) = \mathcal J_{\mathcal V^*}(x^*) + \eta \sum_{j \in [n]}|x^*_j|^2  \leq 1 +\delta.
\end{align*}
This shows that the conditions of Lemma \ref{lem:global_min} are satisfied, and therefore the optimal $\mathcal V^* = \emptyset $ where $q_i(\emptyset) =0 $. Since $(\emptyset,x^*)\in\mathfrak F$, we have $x^*\in\mathcal F_\emptyset$. Compatibility gives $|x_j^*|=1$ for every $j$, while $\emptyset$-admissibility gives $d_I(x^*)\le C$ for every row-direction index $I$, i.e.,
\[
\Big (\xi, \sum_{j \in [n]}a_i(j) x^*_j\Big )_{\R^2} \leq C
\]
for each $\xi \in \{\pm e_1, \pm e_2\}$ and  each row $i \in [m]$.  Therefore, the absolute value of each coordinate of 
\[
\sum_{j \in [n]}a_i(j) x^*_j
\]
is bounded by $C$. Thus, the Euclidean norm of this sum is upper bounded by $C\sqrt {2}$. 

\end{proof}
\begin{remark}
\label{rem:constants} The following constants satisfy Assumptions \ref{as:C} and \ref{as:alpha}:
\[
\alpha = \frac{49}{4}, ~\beta = \frac{1}{2}, ~\lambda = \frac{1}{6},~ C = 46.
\] These constants yield $\rho = 2/3$. Taking  $\delta =1/100$,  the left hand side of Eq. \eqref{eq:1plus_delta_bound} is 
\[
\frac{4 (1 +\delta)}{\alpha}+ \rho = \frac{3662}{3675} <1.
\] 
\end{remark}

\section*{Acknowledgements} V.A.K.  expresses his great appreciation to Alexander Rakhlin for introducing him to the Bellman function method of Burkholder. This was foundational in shaping the conceptual  framework for this project. V.A.K. is deeply grateful to Sinan G\"unt\"urk for sharing his insights into combinatorial discrepancy, which motivated and inspired him  to continue pursuing this project. N.G. is indebted to Deane Yang and Shay Sadovsky for many helpful and interesting discussions. N.G. would like to acknowledge the generous support of the National Science Foundation through grant DMS-2144232. \\

 \emph {Use of Large Language Models (LLMs):} This paper was written up until the first draft, and with a complete proof made fully by the authors, without the use of LLMs. Then, the authors used LLMs to brainstorm ways to upgrade the complex Koml\'os result into one for the original Koml\'os conjecture. This exploration was inconclusive, and we have not incorporated anything produced by this exploration in the present manuscript. We also used LLMs for verifying proofs and proofreading the manuscript. Every suggestion made by an LLM system was independently verified, and corrected when necessary, by the authors, who take full responsibility for the manuscript as written.  

In addition, LLMs were used to create Python code to run numerical experiments. The numerical experiments were meant to check the sharpness (or lack of) in constants for several intermediate lemmas. The outcome of the numerical experiments was compatible with our belief that the constant in the main result can be vastly improved.

N.G. utilized the following commercial LLMs: Claude and GLM 5.2. V.A.K. worked with the following commercial LLMs: ChatGPT (Academic Researchers Program) and Claude. V.A.K.’s work was supported in part by a grant of access to OpenAI models through the ChatGPT for Academic Researchers program.

\bibliographystyle{plainurl}
\bibliography{Komlos}

\appendix

\section{Previously omitted proofs }\label{a:omitted proof of some standard proposition}

In this Appendix, we record the proofs of several results used in the paper. Since these proofs are relatively straightforward,  they have been collected here to avoid distraction from the main ideas of the paper. In what follows, we provide the proofs of  Proposition \ref
{eq:complex_vector_equivalence}, Remarks \ref{rem:c} and \ref{rem:omega}, and Lemmas  \ref{lem:Dz}, \ref{lem:grad_hessians}  and \ref{lem:contraction}. 

\begin{proof}[Proof of  Proposition \ref
{eq:complex_vector_equivalence}]

Let us write an arbitrary complex Koml\'os matrix $C$ in $\mathcal C(m,n)$ as $C=: A + i B$ where $A$ and $B$ are matrices in $\R^{m\times n}$, and define a ``stacked matrix''
\[
D:= \begin{pmatrix} A  \\ B\end{pmatrix}
\]
in $\R^{2m \times n}$, and also denote the $j$-th columns of $A$, $B$ and $D$ by $a_j$, $b_j$ and $d_j$. 
Let 
\[
K_2(m,n):=\max_{A \in \mathcal A(m,n)} \text{Vdisc}_2(A),
\] and note that $D$ is a real Koml\'os matrix  in $\mathcal A(2m,n)$:
\[
|d_j|^2 =  |a_j|^2+ |b_j|^2 = \sum_{i \in [m]} |A_{ij}|^2+|B_{ij}|^2  =  \sum_{i \in [m]} |C_{ij}|^2 \leq 1. 
\]
Therefore, there exist $u_1, \dots, u_n \in \mathbb{S}^1$, such that 
\[
 \max_{k \in [2m]}\Big|\sum_{j \in [n]} D_{kj}u_j\Big| \leq K_2(2m,n)
\]
and then, 
\[
 \max_{i \in [m]}\Big|\sum_{j \in [n]} A_{ij}u_j\Big| \leq K_2(2m,n)~~~\text{and}~~~ \max_{i \in [m]}\Big|\sum_{j \in [n]} B_{ij}u_j\Big| \leq K_2(2m,n).
\]
Using the identification of $u_j$'s with $z_j$ in the text accompanying Eq. \eqref{eq:real_complex_identification}, for any row $i \in [m]$, we have 
\[
\Big|\sum_{j \in [n]} C_{ij}z_j\Big|  \leq \Big|\sum_{j \in [n]} A_{ij}z_j\Big| + \Big|\sum_{j \in [n]} B_{ij}z_j\Big|\leq 2 K_2(2m,n).
\]
The reverse inequality follows from the fact that $\mathcal A(m,n) \subset \mathcal C(m,n)$, and for real matrices we already showed that $\text{cdisc}(A) =\text{Vdisc}_2(A)$ in Eq. \eqref{eq:cdisc_vdisc2}. 
\end{proof}

\begin{proof} [Proof of  Remark \ref{rem:c}]
Since  $0 \leq G \leq q$, and the inequality \eqref{eq:C_inequality} holds, we have
\begin{align*}
z(0,G,q) &= \frac{\beta}{\sqrt {q} } G - C+\omega\sqrt{q} \\
&\leq(\beta+\omega)\sqrt {q}  - C\\
&\leq(\beta+\omega)\sqrt {\alpha}   - C\\
&< 0.
\end{align*}
Also, we have
\begin{align*}
\log \gamma(0, G,q) & = \log \Big(\frac{\alpha}{q}\Big) + \frac{\lambda }{\sqrt {q}}\Big(   \frac{\beta  }{\sqrt {q}}  G - C+ \omega \sqrt {q}\Big)  \\
&\leq \log \Big(\frac{\alpha}{q}\Big) + \lambda\Big(   \beta  - \frac{C}{\sqrt {q}}+\omega\Big) \\
& = 2\log \Big(\sqrt {\frac{\alpha}{q}}\Big) + \lambda (\beta +\omega)   - \frac {C \lambda}{\sqrt {\alpha}}\sqrt {\frac { \alpha}{q}}.    
\end{align*}
We set $y := \sqrt {\alpha / q}$, and observe that $ y \geq 1$.
We define 
\begin{align*}
F(y)  &:= 2 \log y + \lambda\Big ( \beta+\omega - \frac{C} {\sqrt {\alpha}}  \, y \Big),
\end{align*}
and note that
\[
\log \gamma(0, G,q) \leq F(y) ~\text{and} ~ F(1) <0.
\]
Since  $\omega \geq 2/\lambda$, 
\[
F'(y) = \frac{2}{ y} -  \frac{\lambda C} {\sqrt {\alpha}}   \leq 2-\lambda (\beta +\omega) <0,
\]
and therefore
\[
\log \gamma(0, G,q) < 0~\text{and}~ \gamma(0, G,q)<1.
\]

\end{proof}

\begin{proof} [Proof of Remark \ref{rem:omega}] We have
\[
\partial_q z(d,G,q) =  - \frac{\beta}{2q^{\frac{3}{2}}}   G + \frac{\omega}{2\sqrt{q}}  = \frac{1}{2\sqrt{q}} \Big(\omega- \frac{\beta}{q}   G  \Big) \geq\frac{1}{2\sqrt{q}} \Big(\omega- \beta  \Big).
\]
Accordingly, any $\omega > \beta$ guarantees that the left hand side above is strictly positive.  Also, 
\[
\gamma(d, G,q)  = \frac{\alpha}{q} \exp \Big(  \frac{\lambda z(d,G,q)  } {\sqrt {q}}  \Big)  
\]
where $ z(d,G,q)= d + \frac{\beta}{\sqrt {q}}G - C+\omega \sqrt{q}$. Therefore,
\[
\log \gamma(d, G,q)  = \log\alpha - \log q +  \frac{ \lambda z(d,G,q)  } {\sqrt {q}}.   
\]
Since $z \leq 0$ and $G \leq q$,
\begin{align*}
\partial_q \log \gamma(d, G,q) & =  - \frac{1}{ q} +   \frac{\lambda}  {\sqrt {q}}\Big (\partial_q z(d,G,q)   - \frac{1}{2} \frac{ z(d,G,q)  } {q}\Big)   \\
 & \geq  - \frac{1}{ q} +   \frac{\lambda}  {\sqrt {q}}\Big (\frac{1}{2\sqrt{q}} \Big(\omega- \frac{\beta}{q}   G  \Big) \Big)   \\
  & \geq     \frac{\lambda}  {2q}\Big(\omega-  \frac{2}{ \lambda }-\beta \Big)   \\
    &=  0.
\end{align*}
\end{proof}

\begin{proof}[Proof of  Lemma \ref{lem:Dz}] Let us define  $dx_j := v_j-x_j$. Then
\[
\text{d}_I(v) - \text{d}_I(x)= \Big (\xi, \sum_{j \in [n]}a_i(j) dx_j\Big )_{\R^2} = \sum_{j \in \mathcal V} (a_i(j) \xi,  dx_j )_{\R^2}
\]
Also, since
\[
|x_j + dx_j |^2 = |x_j|^2 + 2(x_j, dx_j) + |dx_j|^2
\]
we have 
\[
|v_j|^2 - |x_j|^2 = 2(x_j, dx_j) + |dx_j|^2
\]
 and therefore
 \[
 G_i(v) - G_i(x) = \sum_{j \in [n]}a_i(j)^2 (|x_j|^2 -|v_j|^2 ) = -   \sum_{j \in \mathcal V}a_i(j)^2 (2(x_j, dx_j) + |dx_j|^2). 
 \]
Combining the foregoing expressions,
 \begin{align*}
z_I(v) - z_I(x)  &=\sum_{j \in \mathcal V}  \Big [(a_i(j) \xi,  dx_j )_{\R^2}- \frac{\beta}{\sqrt {q_i} }    a_i(j)^2 (2(x_j, dx_j) + |dx_j|^2) \Big ] \\
&= ( \nabla z_I(x), v-x)_{(\R^2)^s}- \frac{\beta}{\sqrt { q_i (\mathcal V)}}\sum_{j \in \mathcal V} a_{i}(j)^2 |v_j-x_j|^2. 
\end{align*}
Then 
\[
z_I(x)=d_I(x) + \frac{\beta}{\sqrt {q_i}} G_i(x) - B(q_i). 
\]
and the gradient $\nabla z_I(x)$ is given  by \eqref{eq:E_block}. Therefore, the corresponding Hessian is
\[
D^2  z_I    =   \frac{\beta}{\sqrt {q_i}} D^2 G_i~~~~~\text{and}~~~~~D^2 G_i   = -2 \text{diag}_{j \in \mathcal V}(a_i(j)^2 I_2) \in \R^{2s \times 2s}.
\]
Since $d_I(x)$ is linear in $x$, its  Hessian is zero. 
This confirms Eq.  \eqref{eq:D2z}.
\end{proof}

\begin{proof}[Proof of  Lemma \ref{lem:grad_hessians}] Since
\[
z_I(x)=d_I(x) + \frac{\beta}{\sqrt {q_i}} G_i(x) - B(q_i). 
\]
and
\[
\gamma_I(x)  = \frac{\alpha}{q_i} \exp  ( \frac{\lambda}{\sqrt {q_i}} ( d_I(x) + \frac{\beta}{\sqrt {q_i}} G_i(x) - B(q_i) ))  
\]
and the gradients $\nabla z_I(x)$ and $\nabla G_i(x)$ in  $(\R^2)^s$ are given respectively by \eqref{eq:E_block} and 
 \[
(\nabla G_i(x))_j   = - 2a_i(j)^2 x_j\in \R^2. 
\]
for $j \in \mathcal V$. Therefore, the corresponding Hessians are
\[
D^2  z_I    =   \frac{\beta}{\sqrt {q_i}} D^2 G_i~~~~~\text{and}~~~~~D^2 G_i   = -2 \text{diag}_{j \in \mathcal V}(a_i(j)^2 I_2) \in \R^{2s \times 2s}.
\]
Since $d_I(x)$ is linear in $x$, its  Hessian is zero. Then
\begin{align}
\nabla  \gamma_I (x)  =  \frac{\lambda}{\sqrt {q_i}} \gamma_I (x) \nabla z_I(x) 
~\text{and}~ D^2  \gamma_I (x)  =  \frac{\lambda}{\sqrt {q_i}} \gamma_I (x) \Big(\frac{\lambda}{\sqrt{q_i}} \nabla z_I (x)  \nabla z_I (x)^\top+   D^2 z_I\Big). \label{eq:D2gamma}
\end{align}
This confirms Eq.  \eqref{eq:Hessian_gamma}.
\end{proof}
\begin{proof} [Proof of Lemma \ref{lem:contraction}] Since the regularization term does not depend on $\mathcal V$, if $x$ is $\mathcal V$ and $\mathcal V'$-\emph {admissible},  to prove 
\[
\mathcal J_{\mathcal V'}(x) \leq \mathcal J_{\mathcal V}(x).
\]
it is sufficient to prove
\[
\mathcal S_{\mathcal V'}(x) \leq \mathcal S_{\mathcal V}(x).
\]

\textbf{  $\mathcal V'$ is empty.}  Let us denote $q_i \equiv q_i(\mathcal V)$ and $q'_i \equiv q_i(\mathcal V')$ for brevity.  When $\mathcal V'$ is empty, then $q'_i=0$, and we need to check 3 cases:
\[
q_i>\alpha, ~ 0< q_i \leq \alpha, ~\text{and}~ q_i =0.
\]
If $q_i=0$ or $q_i>\alpha$, then $d_I \leq C$ for each row-direction index $I$ containing the corresponding $i$ by the  $\mathcal V$-admissibility of $x$. Therefore, $x$ is $\emptyset$-admissible. If $0 < q_i \leq \alpha$, then  $z_I(x) \leq 0$ for each row-direction index $I$ containing the corresponding $i$. Therefore,  $x$ is also  $\emptyset$-admissible by Remark \ref{rem:medium_disc}. Finally, since each $L^{\mathcal V}_j \geq 0$, we have $\mathcal S_{\emptyset} =0\leq \mathcal S_{\mathcal V}$. 

\textbf{ $\mathcal V'$ is non-empty.} When $\mathcal V'$ is nonempty, it is sufficient to prove that
 for every $I$,
\[
\hat \gamma_{I, \mathcal V'}(x) \leq \hat \gamma_{I, \mathcal V}(x).
\]
Therefore, for every $j \in \mathcal V'$,
\[
L^{\mathcal V'}_j(x) \leq L^{\mathcal V}_j(x).
\]
and our lemma follows by the monotonicity of the softmax.   As noted in Section \ref{sec:BJ},  the Euclidean norm of every $x_j$ for $j \notin \mathcal V'$ is equal to one, and therefore
\[
0 \leq G_i(x) \leq q_i'\leq q_i.
\]
Therefore, for $I$ containing any medium row $i$, $(d_I, G_i(x), q_i)$  automatically belongs to $\Omega$. Also    $d_I(x)$ and $G_i(x)$ do not depend on $\mathcal V$. Therefore, if originally, $i$ is a medium row and $q'_i> 0$, then  $(d_I, G_i(x), q'_i)$  also belongs to $\Omega$. We will consider the same 3 cases as before
\begin{enumerate}
    \item\textbf {Old large row --} {\boldmath $q_i>\alpha$:} When $q_i>\alpha$, we have $d_I=0$. If $q'_i>\alpha$ or $q'_i=0$, $x$ will remain $\mathcal V'$-admissible since $d_I$ does not depend on $\mathcal V$.  Also if $q'_i>\alpha$,  $\hat \gamma_{I, \mathcal V'}(x) = \hat \gamma_{I, \mathcal V}(x)=1$, while  if  $q'_i=0$, then $0=\hat \gamma_{I, \mathcal V'}(x) \leq \hat \gamma_{I, \mathcal V}(x)$. 

If $0<q'_i\leq \alpha$, by Lemma \ref{lem:monotonicity}(a), $z(0, G_i(x), q'_i)<0$, and therefore $x$ will be $\mathcal V'$-admissible. We  also have $\hat \gamma_{I, \mathcal V'}(x) \leq \hat \gamma_{I, \mathcal V}(x)=1$.

\item \textbf{Old medium row --} {\boldmath $0<q_i \leq \alpha$:}  If $0 <q_i \leq \alpha$,  we need to consider 2 subcases: \textbf{(a)} If $0<q'_i\leq q_i$, then $z(d_I(x), G_i(x), q_i)\leq 0$ and $ (d_I(x), G_i(x), q'_i) \in \Omega$. Therefore,  by Lemma  \ref{lem:monotonicity}(c), $z(d_I(x), G_i(x), q'_i)\leq 0$ and $\hat \gamma_{I, \mathcal V'}(x) \leq \hat \gamma_{I, \mathcal V}(x)$.  \textbf{(b)}  Finally, if   $0=q'_i<q_i$, then $z(d_I(x), G_i(x), q_i) \leq 0$, and therefore $d_I(x) \leq C$  by Lemma  \ref{lem:monotonicity}(b). Also, in this case, $0=\hat \gamma_{I, \mathcal V'}(x) \leq \hat \gamma_{I, \mathcal V}(x)$. 

\item\textbf{Old small row --} {\boldmath $q_i=0$:} If $q_i=q'_i =0$, then $d_I(x) \leq C$ and $\hat \gamma_{I, \mathcal V'}(x) = \hat \gamma_{I, \mathcal V}(x)=0$. 
\end{enumerate}
The foregoing shows that if $x$ is $\mathcal V$-admissible, then it will remain $\mathcal V'$-admissible and $\hat \gamma_{I, \mathcal V'}(x) \leq \hat \gamma_{I, \mathcal V}(x)$ in all possible cases, which completes the proof of the lemma.

\end{proof}

\end{document}